\documentclass[twoside,12pt]{amsart}

\usepackage{amsmath,amsthm,amsfonts,amssymb,verbatim,enumerate}
\usepackage{array,charter}
\usepackage{longtable}
\usepackage{verbatim}
\usepackage{pdfpages}

\usepackage{latexsym,mathptm, times,mathcomp,lscape}
   \usepackage{tikz,stmaryrd}
\usepackage{calligra}
\usepackage{calrsfs}
\usepackage[mathscr]{euscript}

\input amssym.def
\input amssym 
\input xypic
\input xy
\xyoption{all}

\usepackage[right=1.5in,left=1.5in,top=1in,bottom=1in]{geometry}
\usepackage{amssymb}
\newcommand{\gf}[2]{\genfrac{}{}{0pt}{}{#1}{#2}}
\def\fakeht{\vphantom{E^{E_E}_{E_E}}}
\def\br{{\bf r}}
\def\bp{{\bf p}}
\def\Phitilde{\widetilde{\Phi}}
\def \transpose{^{\rm T}}
\def \ms{\medskip}

\def \bs{\bigskip}
\def\Z{\mathbb Z}
\def\ts{\textstyle}

\def\B{\mathbb B}
\def\Sym{\operatorname{Sym}}
\def\kk{{\pmb k}}

\def\ann{\operatorname{ann}}
\def\p{\oplus}
\def\Hom{\operatorname{Hom}}
\def\socle{\operatorname{socle}}
\def\t{\otimes}

\def\Pf{\operatorname{Pf}}
\def\a{\alpha}
\def\map{\operatorname{map}}
\def\proj{\operatorname{proj}}
\def\pdelta{\pmb \delta}

\newtheorem{theorem}{Theorem}[section]
\newtheorem{lemma}[theorem]{Lemma}
\newtheorem{observation}[theorem]{Observation}
\newtheorem{proposition}[theorem]{Proposition}
\newtheorem{corollary}[theorem]{Corollary}

\newtheorem{claim}[equation]{Claim}

\theoremstyle{definition}
\newtheorem{data}[theorem]{Data}

\newtheorem{chunk}[theorem]{}
\newtheorem{example}[theorem]{Example}
\newtheorem{remark}[theorem]{Remark}
\newtheorem{remarks}[theorem]{Remarks}
\numberwithin{equation}{theorem}

\begin{document}

\baselineskip=16pt

\title[Quadratically presented Gorenstein ideals]
{Quadratically presented Gorenstein ideals}

\date{\today}

\author[Sabine El Khoury and Andrew R. Kustin]
{Sabine El Khoury and Andrew R. Kustin}

\address{Mathematics Department,
American University of Beirut,
Riad el Solh 11-0236,
Beirut,
Lebanon}
\email{se24\@aub.edu.lb}

\address{Department of Mathematics\\ University of South Carolina\\\newline
Columbia SC 29208\\ U.S.A.} \email{kustin\@math.sc.edu}



\thanks{We used  Macaulay2 \cite{M2} to test our formulas.}

\begin{abstract} 
Let $J$ be a quadratically presented grade three Gorenstein ideal in the standard graded polynomial ring $R=\kk[x,y,z]$, where $\kk$ is a field. Assume that $R/J$ satisfies the weak Lefschetz property. We give the presentation matrix for $J$ in terms of the coefficients of a Macaulay inverse system for $J$. (This presentation matrix is an alternating matrix and $J$ is generated by the maximal order Pfaffians of the presentation matrix.) Our formulas are computer friendly; they involve only matrix multiplication; they do not involve multilinear algebra or complicated summations. As an application, we give the presentation matrix for $J_1=(x^{n+1},y^{n+1},z^{n+1}):(x+y+z)^{n+1}$, when $n$ is even and the characteristic of $\kk$ is zero. Generators for $J_1$ had been identified previously; but the presentation matrix for $J_1$ had not previously been known. The first step in our proof is to give improved formulas for the presentation matrix of a linearly presented grade three Gorenstein ideal $I$ in terms of the coefficients of the Macaulay inverse system for $I$.  
 \end{abstract}

\maketitle

\section{Introduction.}
Let $R=\kk[x,y,z]$ be a standard graded polynomial ring over the field $\kk$,  $K$ be an ideal of $R$   
generated by homogeneous forms of degree $n$ with $R/K$ Artinian and Gorenstein, and  $s$ be the socle degree of $R/K$. 
If the ideal $K$ is linearly presented, then  $s=2n-2$ and the minimal resolution of $R/K$ by free $R$-modules has the form 
\begin{equation}0\to R(-2n-1)\xrightarrow{b_3} R(-n-1)^{2n+1}\xrightarrow{b_2}R(-n)^{2n+1} \xrightarrow{b_1}R.\label{1.0.2}\end{equation}
If the ideal $K$ is quadratically presented, then $n$ is even, $s=2n-1$, and  the minimal resolution of $R/K$ by free $R$-modules has the form 
\begin{equation}0\to R(-2n-2)\xrightarrow{b_3} R(-n-2)^{n+1}\xrightarrow{b_2}R(-n)^{n+1} \xrightarrow{b_1}R.\label{BN}\end{equation}
The theorem of Buchsbaum and Eisenbud \cite{BE77} guarantees that the matrix $b_2$ is an alternating matrix of linear (respectively, quadratic) forms, $b_1$ is the row vector of signed maximal order Pfaffians of $b_2$, and $b_3$ is the transpose of $b_1$.

The socle of the graded Artinian Gorenstein $\kk$-algebra $R/K$ is the one dimensional vector space $(R/K)_s$. Fix an isomorphism $\sigma: (R/K)_s\to \kk$. Let $\Phi_s:R_s\to \kk$ represent the composition
\begin{equation}\label{||}R_s\xrightarrow{\text{natural quotient map}} (R/K)_s\xrightarrow{\sigma} \kk.\end{equation} (The  
homomorphism $\Phi_s\in \Hom_\kk(R_s,\kk)$ is called a 
 Macaulay inverse system for $K$.) 
We give explicit formulas for the alternating matrix $b_2$ in terms of the constants $\{\Phi_s(\mu_s)\}$ as $\mu_s$ roams over the monomials of $R$ of degree $s$. Our formulas involve only matrix multiplication; they do not involve multilinear algebra or complicated summations; and they are computer friendly.
In the  $s=2n-2$ case, no further hypothesis is needed and  our procedure determines if $K$ is linearly presented; and, if so, then gives the precise formula for $b_2$. In the $s=2n-1$ case, our procedure requires that we know a weak Lefschetz element for $R/K$; once this element is identified, then our procedure  determines if $K$ is quadratically presented; and, if so, then gives the precise formula for $b_2$.

Our answer in the linearly presented case is an extension and reformulation of the result in \cite{EKK2}. 
The new version  
is significantly easier to apply than the original version. In particular, we apply the new version  to the $s=2n-1$ case. 

Our main results concern the $s=2n-1$ case. We became interested in this problem in 2010 when we first suspected that if $\kk$ is a field of characteristic zero and $n$ is even, then the grade three Gorenstein  ideal
\begin{equation}\label{1.0.1}J=(x^{n+1},y^{n+1},z^{n+1}):(x+y+z)^{n+1}\end{equation}  
is quadratically presented. In 2010, we produced $n+1$ elements in the ideal  $J$ of degree $n$ and we conjectured that these elements generate $J$. This conjecture is established in \cite[Prop.~5.7]{KRV22}.
 It is easy to show that $x$ is a weak Lefschetz element on $R/J$ and a Macaulay inverse system for $J$ 
 is well known.
 In Section~\ref{sect5} we obtain the presentation matrix for the ideal $J$ of (\ref{1.0.1}) as an application of our main result, Theorem~\ref{main}.

An Artinian standard graded algebra $A$ over a field $\kk$ satisfies the weak Lefschetz property (WLP) property if there is a linear form $\ell$ in $A$ so that, for each index $i$, the map $A_i\to A_{i+1}$, which is given by multiplication by $\ell$, is either injective or surjective. (In this case, $\ell$ is a weak Lefschetz element for $A$.)
It is difficult to determine which graded Artinian $\kk$-algebras satisfy the WLP. The problem has been attacked from many  points of view, including representation theory, topology, vector bundle theory, plane partitions, splines, and differential geometry, among others; 
see for example, \cite{MT,MN, HMMNWW}. It is  not known if all standard graded codimension three Artinian Gorenstein algebras over a field of characteristic zero satisfy the WLP. Indeed, it is not known if all quotient rings $R/J$ resolved by (\ref{BN}) necessarily satisfy the WLP.

\bs
We describe the approach taken in this note. 
\begin{chunk}
\label{approach}
Let $J\subseteq R$ be a homogeneous grade three ideal so that the minimal homogeneous resolution of $R/J$ by free $R$-modules has the  Betti numbers of (\ref{BN}).  
 Let $U$ be the vector space of homogeneous elements of $R$ of degree $1$. (So $R$ is equal to the symmetric algebra $\Sym_\bullet U$.) The divided power algebra $D_\bullet U^*=\bigoplus_{i=0}^\infty D_iU^*$, with $D_iU^*=\Hom_{\kk}(\Sym_iU,\kk)$, is the graded dual of $\Sym_\bullet U$. The algebras $\Sym_\bullet U$ and $D_\bullet U^*$ are modules over one another and the ideal $J$ is equal to $\ann_{\Sym_\bullet U}\Phi_{2n-1}$, for some $\Phi_{2n-1}\in D_{2n-1}U^*$. The element $\Phi_{2n-1}$ in $D_{2n-1}U^*$
is called a Macaulay inverse system of $J$. This element is unique up to multiplication by a non-zero constant of $\kk$. In the case of (\ref{1.0.1}), 
\begin{equation}\Phi_{2n-1}=(x+y+z)^{n+1}\left(
(x^ny^nz^n)^*
\right).\label{PHI}\end{equation} The symbols and module action in 
(\ref{PHI}) 
 are explained in \ref{x^*}.
\end{chunk}

The hypothesis that $R/J$ satisfies the WLP ensures that there is a non-zero element $\ell$ of $U$ for which 
the ideal $I=\ann_{\Sym_{\bullet}U}\ell(\Phi_{2n-1})$ of $R$  
is a linearly presented grade three Gorenstein  ideal. We choose a basis $x,y,z$ for $U$ with $x=\ell$; and we let $U_0$ be the subspace of $U$ which is spanned by $y$ and $z$. The minimal homogeneous resolution of $R/I$ by free $R$-modules is given in \cite{EKK2}. This resolution has the form 
\begin{equation}\notag 
 \B:\quad 0\to B_3\xrightarrow{b_3} B_2\xrightarrow{b_2} B_1\xrightarrow{b_1} B_0,\end{equation} with
$B_3=B_0=R$,  
$$B_2=\left\{\begin{matrix} R\t_{\kk} \Sym_{n-1}U_0\\\p\\R\t_{\kk} D_nU_0^*\end{matrix}\right.\quad \text{and}\quad B_1=\left\{\begin{matrix} R\t_{\kk} D_{n-1}U_0^*\\\p\\R\t_{\kk} \Sym_nU_0,\end{matrix}\right.$$
The matrix for $b_2$ looks like 
\begin{equation}\label{nny}\left[\begin{array}{c|c} A&B\\\hline\fakeht-B\transpose&D\end{array}\right],\end{equation} with $A$ and $D$ alternating matrices, $A=xA'$ where $A'$ is an $n\times n$ alternating matrix of constants. The graded Betti numbers of $\B$ are given in (\ref{1.0.2}).
Explicit formulas for the differentials of $\B$ are given in Theorem~\ref{NewTheorem}.

The hypothesis that the ideal $J$ is quadratically presented forces 
the matrix $A'$ 
to be  invertible. It follows quickly  that
 the resolution of $R/J$ is
$$0\to R\xrightarrow{{b_1'}^*}
R\t_{\kk} D_nU_0^* \xrightarrow{B\transpose (A')^{-1}B+xD}
R\t_{\kk} \Sym_nU_0\xrightarrow{b_1'} R,$$
where $b_1'$ is the restriction of $b_1$ to the summand  $R\t_{\kk} \Sym_nU_0$ of $B_1$. 
 The parity of $n$ plays a critical role because, if $n$ had been  odd, then the 
 $n\times n$ alternating matrix of constants $A'$ would 
 necessarily be singular.

\tableofcontents

\section{Notation, conventions, and elementary results.}\label{Prelims}

In this paper ${\kk} $ is always a field and  the symbol ``$D_iG^*$'' always means $D_i(G^*)$.  See \cite[Thms.~21.5 and 21.6]{E95} for a quick treatment of Artinian Gorenstein rings and Macaulay duality. 

\begin{chunk} When it is clear that ``A'' is the ambient ring, we use 
$(-)^*$, $\Sym$, $D$, $\t$, $:$, and $\ann$ to mean $\Hom_A(-,A)$, $\Sym^A$, $D^A$, $\t_A$, $:_A$, and $\ann_A$, respectively. \end{chunk}

\begin{chunk} The {\it grade} of an ideal $I$ in a commutative Noetherian ring $R$ is the length of a
maximal regular sequence on $R$ which is contained in $I$.\end{chunk}

\begin{chunk} 
If $M$ is a graded module, then $M_i$ denotes the component of $M$ of  homogeneous elements of degree $i$ and $M=\oplus_i M_i$. \end{chunk}

\begin{chunk} 
For any set of variables $\{x_1,\dots,x_r\}$ and any degree $s$, we write
$\binom{x_1,\dots,x_r} s$ for   the set of monomials of degree $s$ in the variables $x_1,\dots,x_r$.\end{chunk}

\begin{chunk}
\label{x^*} We explain the symbols and module action of $$(x+y+z)^{n+1}\left(
(x^ny^nz^n)^*
\right)
$$from (\ref{PHI}).

If $A$ is a commutative ring and $G$ is a finitely generated free $A$-module, then $D^A_\bullet G$ is the graded dual of the $A$-module $\Sym^A_\bullet G$. In other words, 
$$D^A_\bullet G=\bigoplus_{i=0}^\infty D^A_iG^*$$and $D^A_iG^*=\Hom_A(\Sym^A_iG,A)$.
The $A$-module $D^A_\bullet G$ is a $\Sym^A_\bullet G$-module under the action $u_i\in\Sym^A_iG$ sends $w_j\in D^A_jG^*$ to the element $u_i(w_j)$ in $D^A_{j-i}G^*$,  where 
$u_i(w_j)$ sends $u_{j-i}$ to $(u_{j-i}u_i)(w_j)$, for $u_{j-i}$ in $\Sym^A_{j-i}G$.
If $x_1,\dots,x_g$ is basis for $G$, then 
$$\binom{x_1,\dots,x_g}i =\text{ the set of monomials of degree $i$ in $x_1,\dots,x_g$}$$ is a basis for $\Sym^A_iG$, and $\{m^*\mid m\in \binom{x_1,\dots,x_g}i
\}$ is a basis for $D_iG^*$, where for each $m\in \binom{x_1,\dots,x_g}i$,
$m^*:\Sym^A_iG\to A$ is defined by
$$m^*(m')=\begin{cases} 1,&\text{if $m'=m$,}\\0,&\text{if $m'\neq m$},\end{cases}$$ for $m'\in \binom{x_1,\dots,x_g}i$.
Observe that $$x_\ell(m^*)=\begin{cases} 0,&\text{if $x_\ell\not| m$},\\
(\frac m {x_\ell})^*, &\text{if $x_\ell| m$},\end{cases}$$
for $m\in \binom{x_1,\dots,x_g}i$.
\end{chunk}

\begin{chunk}
The graded ring $R=\bigoplus
_{0\le i}R_i$ is a {\it standard graded $R_0$-algebra} if $R$ is generated as an $R_0$-algebra by $R_1$ and $R_1$ is finitely generated as an $R_0$-module. 
\end{chunk}

\begin{chunk}
\label{2.7}
 If $M$ is a matrix, then $M\transpose$ is the transpose of $M$. The matrix $M$ is an {\it alternating matrix} if $M+M\transpose=0$ and the entries of $M$ on the main diagonal are all zero. 
Let $M$ be an $m\times m$ alternating matrix. The Pfaffian of $M$ is a square root of a determinant of $M$. (For a more detailed formulation see, for example \cite[3.7]{KPU17} or \cite[Appendix and Sect.~2]{BE77}.) If $m$ is odd, 
then the Pfaffian of $M$ is zero and the   {\it row vector of signed maximal order Pfaffians of $M$} is
$$\bmatrix M_1&\dots&M_m\endbmatrix,$$ where $M_j$ is equal to $(-1)^{j+1}$ times the Pfaffian of $M$ with row and column $j$ deleted.
The product $\bmatrix M_1&\dots&M_m\endbmatrix M$ is equal to zero.
\end{chunk} 

\begin{chunk}
\label{order} Let $U$ be the vector space, over the field $\kk$,  with basis $x,y,z$; and let $U_0$ be the subspace of $U$ spanned by $y$ and $z$. 
We make matrices for $\kk$-module homomorphisms between $\Sym_iX$ and $D_iX^*$ for various $i$ and for $X$ equal to either $U$ or $U_0$. Here are the bases we use.
 
If 
$i$ is a positive integer, then our favorite ordered basis for $\Sym_iU$ is 
$$x^i,\ x^{i-1}y,\ x^{i-1}z,\ x^{i-2}y^2,\ x^{i-2}yz,\ x^{i-2}z^2,\ \dots,\ y^i,\ y^{i-1}z,\ \dots,\ yz^{i-1},\ z^i.$$ In other words, if $a+b+c$ and $\a+\beta+\gamma$ are both equal to $i$, then $x^ay^bz^c$ comes before $x^\a y^\beta z^\gamma$ if 
$$\a<a;\quad\text{or else,}\quad \a=a\text{ and }\beta<b.$$
In particular, if $m_1,\ \dots,\ m_{\binom {i+1}2}$ is our favorite basis for $\Sym_{i-1}U$, then our favorite basis for $\Sym_iU$ is 
\begin{equation}\label{favorite}xm_1,\ \dots,\ xm_{\binom {i+1}2},\ y^i,\ y^{i-1}z,\ \dots,\ yz^{i-1},\ z^i;\end{equation} our favorite basis for $D_{i-1}U^*$ is
 $(m_1)^*,\ \dots,\ (m_{\binom {i+1}2})^*$; our favorite ordered basis for $\Sym_iU_0$ is   $$y^i,\ y^{i-1}z,\ \dots,\ yz^{i-1},\ z^i;$$ and our favorite ordered basis for $D_iU_0^*$ is $$(y^i)^*,\ (y^{i-1}z)^*,\ \dots,\ (yz^{i-1})^*,\ (z^i)^*.$$
\end{chunk}

\begin{chunk}Let $I$ be an ideal in a ring $A$, $N$ be an $A$-module, and  $L$ and $M$ be submodules of $N$. Then 
$$L:_IM=\{x\in I\mid xM\subseteq L\}\quad\text{and}\quad L:_MI=\{m\in M\mid Im\subseteq L\}.$$If $L$ is the zero module, then we also use ``annihilator notation'' to describe these ``colon modules''; that is,
$$\ann_AM=0:_AM\quad\text{and}\quad \ann_NI=0:_NI.$$
\end{chunk}

\begin{chunk}
If $A$ is a local ring with maximal ideal $\mathfrak m$ and $M$ is an $A$-module, then the {\it socle} of $M$ is the vector space $\socle(M)=0:_M\mathfrak m$.
If $A$ is a graded Artinian local Gorenstein $\kk$-algebra, then 
$$s=\max\{i|A_i\neq 0\}$$ is called the {\it socle degree} of $A$.
\end{chunk}

\begin{chunk}
\label{J18} 
Let $R$ be a Noetherian graded $\kk$-algebra. A finitely generated $R$-module $M$ is {\it linearly presented} if the minimal homogeneous resolution of $M$ by free $R$-modules has the form 
$$\dots\to F_2 \xrightarrow{d_2}F_1\xrightarrow{d_1} F_0,$$ and all of the entries of the matrix $d_2$ are linear forms in $R$.

The following result gives many conditions which are equivalent to the statement ``the grade three Gorenstein ideal $I$ is linearly presented''. The initial hypotheses of the  
Proposition 
 ensure that $\Phi_s$ is a Macaulay inverse system for $I$ and that the socle degree of $R/I$ is $s$.
\begin{proposition}
\label{1.8}
  {\bf \cite[Prop.~1.8]{EKK1}} Fix positive integers $n$ and $s$. Let $\pmb k$ be a field, 
$U$ be a three dimensional vector space over $\kk$, $\Phi_s$ be an element of $D_{s}^{\pmb k}U^*$, and $I$ be the ideal $\ann_{\Sym U}\Phi_s$ in the standard graded  polynomial ring $R=\Sym^\kk_\bullet U$.  
 Then the  following statements are equivalent\,{\rm:}
\begin{enumerate}[\rm(a)]
\item 
$I$ is generated by homogeneous forms of degree $n$ and $I$ is linearly presented,  
\item the minimal homogeneous resolution of $R/I$ by free $R$-modules has the form
$$0\to R(-2n-1)\to R(-n-1)^{2n+1}\to R(-n)^{2n+1}\to R,$$
\item all of the minimal generators of $I$ have degree  $n$  and 
 $s=2n-2$,
\item $I_{n-1}=0$ and $(R/I)_{2n-1}=0$,
\item $s=2n-2$ and the homomorphism $p:\Sym_{n-1}^{\pmb k}U\to D_{n-1}^{\pmb k}U^*$, which is given by $$p(\mu_{n-1})=\mu_{n-1}(\Phi_s),\quad\text{for $\mu_{n-1}\in \Sym_{n-1}U$},$$ 
 is an isomorphism, and 
\item\label{2.12.f} 
$s=2n-2$  
and the  $\binom{n+1}2 \times \binom{n+1}2$ matrix
$$(\Phi_s(m_im_j))_{1\le i,j\le \binom{n+1}2},$$ is invertible,
where $m_1,\dots m_{\binom{n+1}2}$ is the ordered basis of $\Sym_{n-1}U$ which is given in {\rm\ref{order}}.
\end{enumerate}
\end{proposition}
The matrix of item (\ref{2.12.f}) has the element $\Phi_s(m_im_j)$ of $\kk$ in row $i$ and column $j$.
\end{chunk}

\section{A computer friendly form of \cite{EKK2}.}
\label{B}

In a sequence of three papers, we gave explicit formulas for the differentials in the minimal homogeneous resolution of $R/I$ by free $R$-modules, where $R=\kk[x_1,\dots,x_g]$ is a standard-graded polynomial ring over the field $\kk$ and $I$ is a grade $g$ homogeneous Gorenstein ideal, provided the resolution is Gorenstein-linear in the sense that the resolution has the form
$$0\to F_g\xrightarrow{f_g}F_{g-1}\xrightarrow{f_{g-1}}F_{g-2}\xrightarrow{f_{g-2}}\dots\xrightarrow{f_3}F_{2}\dots\xrightarrow{f_2}F_{1}\dots\xrightarrow{f_1}R$$
and all the entries of all of the matrices $f_i$, with $2\le i\le g-1$, are linear. The formulas are given in terms of the coefficients of a Macaulay inverse system for $I$. 
The  paper \cite{EKK1} proves that the project can be done;  \cite{EKK2} carries out the project when $g=3$, and \cite{EKK3} carries out the project for all $g$.

 Theorem~\ref{NewTheorem}, which is adapted from \cite[Thm.~4.1]{EKK2}, is the starting point for the present paper.

\begin{theorem}
\label{NewTheorem}  Let $\kk$ be a field, $U$ be a vector space over $\kk$ of dimension $3$ with basis $x,y,z$, and $U_0$ be the subspace of $U$ spanned by $y$ and $z$. Let $n$ be a positive integer, $\Phi_{2n-1}$ be an element of $D_{2n-1}U^*$, $R$ be the polynomial ring $\Sym_\bullet ^\kk U$, and $I$ be the ideal $$I=\ann_{\Sym U}(x\Phi_{2n-1})$$ of $R$. Let $$m_1,\dots,m_{\binom{n+1}2}\quad\text{and}\quad m_{0,1},\dots,m_{0,n+1}$$ be the ordered bases of $\Sym_{n-1}U$ and  $\Sym_nU_0^*$, respectively, which are given in {\rm\ref{order}}. 
Let $\bp$ and $\br$ be the matrices
$$\bp=(xm_i m_{j}(\Phi_{2n-1}))_{1\le i,j \le \binom{n+1}2}
\quad\text{and}\quad
\br=(m_i m_{0,j}(\Phi_{2n-1}))_{\gf{1\le i\le \binom{n+1}2}{1\le j\le n+1}}.$$
Then the following statements hold.
\begin{enumerate}[\rm (a)]
\item The ideal $I$ is linearly presented if and only if the matrix $\bp$ is invertible.
\item\label{NewTheorem.b} If $\bp$ is invertible, then 
 the minimal homogeneous resolution of $R/I$ 
by free $R$-modules 
is 
$$0\to R(-2n-1)\xrightarrow{b_3} R^{2n+1}(-n-1)\xrightarrow{b_2}R^{2n+1}(-n)\xrightarrow{b_1}R,$$
where 
$$b_2=\left[\begin{array}{c|c} A& B\\\hline\fakeht
-B\transpose&D\end{array}\right],$$
$b_1$ is the row vector of signed maximal order Pfaffians of $b_2$,  $b_3=b_1\transpose$,
$$A=xA',\quad A'=A_0-A_0\transpose,\quad  B=xB_1+B_2,\quad 
D=x(D_0-D_0\transpose),$$
\begingroup\allowdisplaybreaks\begin{align*}
A_0={}&([\br\transpose \bp^{-1}\br]
\text{ \rm with row $n+1$  and 
column $1$ deleted}),\\ B_0={}&\text{\rm the right most $n$ columns of $\br\transpose\bp^{-1}$},\\
B_1={}&\left[\begin{array}{c|c}0_{n\times 1}& \begin{array}{c}\text{\rm $B_0$ with}\\\text{\rm row $n+1$}\\ \text{\rm deleted}\end{array}
\end{array}\right]-\left[\begin{array}{c|c} \begin{array}{c}\text{\rm $B_0$ with}\\
\text{\rm row $1$}\\\text{\rm deleted}\end{array}& 0_{n\times 1} \end{array}\right],\\
B_2={}&\left[\begin{array}{c|c}zI_n&0_{n\times 1}\end{array}\right]
-      \left[\begin{array}{c|c}0_{n\times 1}&yI_n\end{array}\right],\\
D_0={}&\left[ \begin{array}{c|c} 0_{n\times 1}&\text{\rm the bottom right $n\times n$ submatrix of $\bp^{-1}$}\\\hline
0_{1\times 1}&0_{1\times n}\end{array}
\right].\end{align*}
\endgroup
\end{enumerate}
\end{theorem}
\begin{remarks} 
\begin{enumerate}
[\rm(a)]
\item We write $I_a$ for the $a\times a$ identity matrix and $0_{a\times b}$ for the zero matrix with $a$ rows and $b$ columns. 
\item The entry of $\bp$ in row $i$ and column $j$ is the element
$$xm_im_j(\Phi_{2n-1})=\Phi_{2n-1}(xm_im_j)$$ of $\kk$; similarly, 
the entry of $\br$ in row $i$ and column $j$ is the element
$$m_im_{0,j}(\Phi_{2n-1})=\Phi_{2n-1}(m_im_{0,j})$$ of $\kk$.
\item The element $x\Phi_{2n-1}\in D_{2n-2}U^*=\Hom_{\kk}(\Sym_{2n-2}U,\kk)$ is a Macaulay inverse system for $I$. If one starts with a Macaulay inverse system $\Phi_{2n-2}$ for $I$, then one may take $\Phi_{2n-1}$ to be any element of $D_{2n-1}U^*$ with $x\Phi_{2n-1}=\Phi_{2n-2}$. In particular, if
$$\Phi_{2n-2}=\sum_{m\in \binom {x,y,z}{2n-2}}\a_m m^*,$$ with $\a_m\in \kk$, then one may take $\Phi_{2n-1}$ to be
\begin{equation}\label{Phi2n-1}\sum_{m\in \binom {x,y,z}{2n-2}}\a_m (xm)^*;\end{equation}however one is not required to make this choice.
\end{enumerate}
\end{remarks}

\begin{table}
\begin{center}
$$\begin{array}{|c|c|c|}\hline
\text{matrix}&\text{number of rows}&\text{number of columns}\\\hline
\br&\binom{n+1}2&n+1\\\hline
\bp&\binom{n+1}2&\binom{n+1}2\\\hline
A&n&n\\\hline
B&n&n+1\\\hline
D&n+1&n+1\\\hline
B_0&n+1&n\\\hline
A_0&n&n\\\hline
B_1&n&n+1\\\hline
B_2&n&n+1\\\hline
D_0&n+1&n+1\\\hline
b_1&1&2n+1\\\hline
b_2&2n+1&2n+1\\\hline
b_3&2n+1&1\\\hline
\end{array}$$
\caption{{The shapes of the matrices of Theorem~\ref{NewTheorem}.}} 
\end{center}
\end{table}

\bigskip
Theorem~\ref{NewTheorem} is obtained from \cite{EKK2} in three steps. 
Theorem~\ref{3.2} is  
the multilinear algebra version of the result from \cite{EKK2} (from a different point of view); Theorem~\ref{3.3} is a crucial extension of Theorem~\ref{3.2}; and Theorem~\ref{NewTheorem} is a matrix version of Theorem~\ref{3.3}. The conversion from the language of Theorem~\ref{3.3} to the language of Theorem~\ref{NewTheorem} is given in \ref{proof}.
 Theorem~\ref{NewTheorem} is easier to apply than anything in \cite{EKK2}.

\bigskip
 There are two differences between the resolution as given in \cite{EKK2} and the resolution of Theorem~\ref{3.2}. 
In \cite{EKK2} the coefficients of a Macaulay inverse system are treated as variables and are specialized to field elements only at the last step. In the present paper, we work in $\kk[x,y,z]$ rather than
$$\mathfrak R=\Z[x,y,z,\{t_{m}\mid m\text{ is a monomial in $x$, $y$, $z$ of degree $2n-2$}\}].$$ In \cite{EKK2}, the map $p$ of Data~\ref{3.1} is a map of free $\mathfrak R$-modules. Each entry of the matrix for $p$ is a variable $t_m$ 
 of the polynomial ring $\mathfrak R$. In \cite{EKK2}, the map $p$ does not have an inverse until $\pdelta$ (which is essentially the determinant of $p$) is inverted. In the present paper, $p$ is an isomorphism of vector spaces; hence, $p$ has an inverse. We write $p^{-1}$ and avoid $q$ (which is essentially the classical adjoint of $p$) and $\pdelta$. Theorem~{4.1} of \cite{EKK2}, with the above two changes of point of view, is recorded as Theorem~\ref{3.2}. 

\bigskip In fact, we make one more change to \cite[4.1]{EKK2} before we can use it in the present situation. The paper \cite{EKK2} makes use of the ``$\Phitilde_{2n-1}$'' of Data~\ref{3.1}; however, \cite{EKK2} makes no mention of the $\Phi_{2n-1}$ of Data~\ref{3.1}. This last fact is unfortunate from our point of view, because $\Phi_{2n-1}$ is a Macaulay inverse system for the quadratically presented ideal $J$, which is the topic of of study in the present paper. In Theorem~\ref{3.3}, we prove that 
if one uses $\Phi_{2n-1}$ in place of $\Phitilde_{2n-1}$ in the maps of Theorem~\ref{3.2}, then one obtains a resolution of $R/I$ which is isomorphic to the resolution of Theorem~\ref{3.2}. We emphasize that Theorem~\ref{3.3} is an extension of Theorem~\ref{3.2}; it is more than a change of notation.
The element of (\ref{Phi2n-1}) is the ``$\Phitilde_{2n-1}$'' of Data~\ref{3.1}. The technique of \cite{EKK2} becomes more valuable when one allows $\Phi_{2n-1}$ to be any element of $D_{2n-1}U^*$ with $x(\Phi_{2n-1})=\Phi_{2n-2}$ rather than insisting  that $\Phi_{2n-1}$ be the element ``$\Phitilde_{2n-1}$'' of
(\ref{Phi2n-1}) and Data~\ref{3.1}.

\begin{data}
\label{3.1} Let $\kk$ be a field, $U$ be a vector space over $\kk$ of dimension $3$, $R$ be the polynomial ring $R=\Sym_\bullet U$, and $n$ be a positive integer.
Decompose $U$ as $\kk x\p U_0$ where $x$ is a non-zero element of $U$ and $U_0$ is a two-dimensional subspace of $U$. Let $\Phi_{2n-1}$ be an element of $D_{2n-1}U^*$. Assume that the homomorphism $$p:\Sym_{n-1}U\to D_{n-1}U^*,$$ which is given by $$p(\mu_{n-1})= \mu_{n-1}(x(\Phi_{2n-1})),\quad\text{for 
$\mu_{n-1}\in \Sym_{n-1}U$},
$$ is an isomorphism.
Let $\rho_{0,2n-1}$ be the element of $D_{2n-1}U_0^*$ with 
$$\mu_{0,2n-1}(\Phi_{2n-1}-\rho_{0,2n-1})=0,\quad\text{for all $\mu_{0,2n-1}\in \Sym_{2n-1}U_0$.}$$  Let 
$$\Phitilde_{2n-1}=\Phi_{2n-1}-\rho_{0,2n-1}\in D_{2n-1}U^*.$$
Let $I$  
be the following ideal
 of $R$: 
$$I=\{\mu\in R\mid \mu(x(\Phi_{2n-1}))=0\}
.$$\end{data}

\begin{remarks} 
\begin{enumerate}[\rm(a)]
\item The hypothesis that the homomorphism $p$ of Data~\ref{3.1} is an isomorphism is equivalent to the hypothesis that $I$ is linearly presented; see, Proposition~\ref{1.8}.
\item The $\widetilde{\B}$ of Theorem~\ref{3.2} may be found 
as \cite[Def.~2.7]{EKK2}. Other versions of $\widetilde{\B}$ may also be found as \cite[Observation~4.4]{EKK2},
or \cite[Proposition~5.5]{EKK2}, or \cite[Definition~3.1]{EKK3}. 
\end{enumerate}
\end{remarks}

\begin{theorem} 
\cite[Thm.~4.1]{EKK2}
\label{3.2}
Adopt Data~{\rm\ref{3.1}}. One minimal homogeneous resolution of $
R/I$ by free $R$-modules is
\begin{equation}\notag \widetilde{\B}:\quad 0\to B_3\xrightarrow{\widetilde{b}_3} B_2\xrightarrow{\widetilde{b}_2} B_1\xrightarrow{\widetilde{b}_1} B_0\end{equation} with
$B_3=B_0=R$,  
$$B_2=\left\{\begin{matrix} R\t \Sym_{n-1}U_0\\\p\\R\t D_nU_0^*\end{matrix}\right.\quad \text{and}\quad B_1=\left\{\begin{matrix} R\t D_{n-1}U_0^*\\\p\\R\t \Sym_nU_0,\end{matrix}\right.$$
\begingroup\allowdisplaybreaks
\begin{align*}
\widetilde{b}_1(\nu_{0,n-1})={}&xp^{-1}(\nu_{0,n-1}),\\ 
\widetilde{b}_1(\mu_{0,n})={}&\mu_{0,n}-xp^{-1}(\mu_{0,n}(\Phitilde_{2n-1})),\\ 
\widetilde{b}_2\bmatrix \mu_{0,n-1}\\0\endbmatrix={}&\begin{cases} 
+x \sum\limits_{m_1\in \binom{y,z}{n-1}} \Big(p^{-1}\Big[(z\mu_{0,n-1})(\Phitilde_{2n-1}) \Big]\Big)\Big[(ym_1)(\Phitilde_{2n-1})\Big]\t m_1^*\vspace{5pt}\\
-x \sum\limits_{m_1\in \binom{y,z}{n-1}}\Big(p^{-1}\Big[(y\mu_{0,n-1})(\Phitilde_{2n-1}) \Big]\Big)\Big[(zm_1)(\Phitilde_{2n-1})\Big]\t m_1^*\\\hline
+x \sum\limits_{m_2\in \binom{y,z}{n}} \Big(p^{-1}\Big[(z\mu_{0,n-1})(\Phitilde_{2n-1}) \Big]\Big)(y(m_2^*))\t m_2\\
-x \sum\limits_{m_2\in \binom{y,z}{n}}\Big(p^{-1}\Big[(y\mu_{0,n-1})(\Phitilde_{2n-1}) \Big]\Big)(z(m_2^*))\t m_2
\\
+y\t z\mu_{0,n-1}
-z\t y\mu_{0,n-1},
\end{cases}\\
\widetilde{b}_2\bmatrix 0\\\nu_{0,n}\endbmatrix ={}& \begin{cases}
-x \sum\limits_{m_1\in \binom{y,z}{n-1}}\Big(p^{-1}\Big[(zm_1)(\Phitilde_{2n-1}) \Big]\Big)(y(\nu_{0,n}))\t m_1^*\\
+x \sum\limits_{m_1\in \binom{y,z}{n-1}} \Big(p^{-1}\Big[(ym_1)(\Phitilde_{2n-1}) \Big]\Big)(z(\nu_{0,n}))\t m_1^*\\
-y\t z(\nu_{0,n})
+z\t y(\nu_{0,n})\\
 \hline
+x \sum\limits_{m_2\in \binom{y,z}{n}}
[p^{-1}(z(\nu_{0,n}))][y(m_2^*)]\t
 m_2\\
-x \sum\limits_{m_2\in \binom{y,z}{n}}
[p^{-1}(y(\nu_{0,n}))][z(m_2^*)]
\t m_2,
\end{cases}\\
\intertext{and}
\widetilde{b}_3(1)={}&\bmatrix \sum\limits_{m_1\in \binom{y,z}{n-1}}\widetilde{b}_1(m_1^*)\t m_1\\
\sum\limits_{m_2\in \binom{y,z}{n}}\widetilde{b}_1(m_2)\t m_2^*\endbmatrix,
\end{align*}\endgroup
for $\mu_{0,i}\in \Sym_i U_0$ and $\nu_{0,i}\in D_i U_0^*$.
\end{theorem}
The graded Betti numbers of $\widetilde{\B}$ from Theorem \ref{3.2}, and also the graded Betti numbers of $\B$ from Theorem \ref{3.3}, are given in (\ref{1.0.2}).
\begin{theorem} 
\label{3.3} Adopt Data~{\rm\ref{3.1}}. A second minimal homogeneous resolution of $R/I$ by free $R$-modules is
\begin{equation}\label{6.5.1} \B:\quad 0\to B_3\xrightarrow{b_3} B_2\xrightarrow{b_2} B_1\xrightarrow{b_1} B_0,\end{equation} 
with $B_i$ as given in Theorem~{\rm\ref{3.2}},
\begingroup\allowdisplaybreaks
\begin{align*}
b_{1}(\nu_{0,n-1})={}&xp^{-1}(\nu_{0,n-1}), 
\\
b_{1}(\mu_{0,n})={}&\mu_{0,n}-xp^{-1}(\mu_{0,n}(\Phi_{2n-1})), \\
b_2\bmatrix \mu_{0,n-1}\\0\endbmatrix ={}&\begin{cases}
+x \sum\limits_{m_1\in \binom{y,z}{n-1}}
\Big(p^{-1}\Big[(z\mu_{0,n-1})(\Phi_{2n-1}) \Big]\Big)\Big[(ym_1)
(\Phi_{2n-1})\Big]\t m_1^*
\\
-x \sum\limits_{m_1\in \binom{y,z}{n-1}}
\Big(p^{-1}\Big[(y\mu_{0,n-1})(\Phi_{2n-1}) \Big]\Big)\Big[(zm_1)(\Phi_{2n-1})\Big]\t m_1^*
\\
\hline
+x \sum\limits_{m_2\in \binom{y,z}{n}}
\Big(p^{-1}\Big[(z\mu_{0,n-1})(\Phi_{2n-1}) \Big]\Big)(y(m_2^*))
\t m_2
\\
-x
\sum\limits_{m_2\in \binom{y,z}{n}}
 \Big(p^{-1}\Big[(y\mu_{0,n-1})(\Phi_{2n-1}) \Big]\Big)(z(m_2^*))
\t m_2 
\\
+y\t z\mu_{0,n-1} 
-z\t y\mu_{0,n-1}, 
\\
\end{cases}\\
b_2\bmatrix 0\\\nu_{0,n}\endbmatrix ={}& \begin{cases}
-x \sum\limits_{m_1\in \binom{y,z}{n-1}}\Big(p^{-1}\Big[(zm_1)(\Phi_{2n-1}) \Big]\Big)(y(\nu_{0,n}))\t m_1^*\\
+x \sum\limits_{m_1\in \binom{y,z}{n-1}} \Big(p^{-1}\Big[(ym_1)(\Phi_{2n-1}) \Big]\Big)(z(\nu_{0,n}))\t m_1^*\\
-y\t z(\nu_{0,n})
+z\t y(\nu_{0,n})\\
 \hline
+x \sum\limits_{m_2\in \binom{y,z}{n}}
[p^{-1}(z(\nu_{0,n}))][y(m_2^*)]
\t m_2\\
-x \sum\limits_{m_2\in \binom{y,z}{n}}
[p^{-1}(y(\nu_{0,n}))][z(m_2^*)]
\t m_2,
\end{cases}\\
\intertext{and}
b_3(1)={}&\bmatrix \sum\limits_{m_1\in \binom{y,z}{n-1}}b_1(m_1^*)\t m_1\\
\sum\limits_{m_2\in \binom{y,z}{n}}b_1(m_2)\t m_2^*\endbmatrix,
\end{align*}\endgroup
for $\mu_{0,i}\in \Sym_iU_0$ and $\nu_{0,i}\in D_iU_0^*$.
\end{theorem}

\begin{proof} It is not difficult to see that
$$\xymatrix{
0\ar[r]&B_3\ar[r]^{b_3}\ar[d]^{=}&B_2\ar[r]^{b_2}\ar[d]^{\theta_2}&B_1\ar[r]^{b_1}\ar[d]^{\theta_1}&B_0\ar[d]^{=}\\
0\ar[r]&B_3\ar[r]^{\widetilde{b}_3}&B_2\ar[r]^{\widetilde{b}_2}&B_1\ar[r]^{\widetilde{b}_1}&B_0,
}$$
with \begingroup\allowdisplaybreaks\begin{align*}\theta_1\left(\bmatrix \nu_{0,n-1}\\\mu_{0,n}\endbmatrix\right)={}&\bmatrix \nu_{0,n-1}-\mu_{0,n}(\rho_{0,2n-1})\\\mu_{0,n}\endbmatrix\quad\text{and}\\
\theta_2\left(\bmatrix \mu_{0,n-1}\\\nu_{0,n}\endbmatrix\right)={}&
\bmatrix \mu_{0,n-1}\\\mu_{0,n-1}(\rho_{0,2n-1})+\nu_{0,n}\endbmatrix,\end{align*}\endgroup
for $\mu_{0,i}\in \Sym_iU_0$ and $\nu_{0,i}\in D_iU_0^*$,
is an isomorphism of complexes.
\end{proof}

\begin{chunk}\label{proof}{\it The proof of Theorem~{\rm\ref{NewTheorem}}.}
The map $b_2$ of Theorem~\ref{3.3} is $$b_2\bmatrix \mu_{0,n-1}\\\nu_{0,n}\endbmatrix =
\bmatrix x\t f_1(\mu_{0,n-1})+x\t f_2(\nu_{0,n})+f_3(\nu_{0,n})\\
x\t f_4(\mu_{0,n-1})+f_5(\mu_{0,n-1})+x\t f_6(\nu_{0,n})\endbmatrix,$$
where
\begin{align*}
f_1:&\Sym_{n-1}U_0\to D_{n-1}U_0^*,&f_2:&D_nU_0^*\to D_{n-1}U_0^*,\\
f_3:&D_nU_0^*\to R\t_\kk D_{n-1}U_0^*&f_4:&\Sym_{n-1}U_0\to \Sym_nU_0\\
f_5:&\Sym_{n-1}U_0\to R\t_\kk \Sym_nU_0,\text{ and}&f_6:&D_nU_0^*\to \Sym_nU_0,\\
\end{align*}are given by
\begingroup\allowdisplaybreaks
\begin{align*}
f_1(\mu_{0,n-1}) ={}&\begin{cases}
\phantom{+}\sum\limits_{m_1\in \binom{y,z}{n-1}}
\Big(p^{-1}\Big[(z\mu_{0,n-1})(\Phi_{2n-1}) \Big]\Big)\Big[(ym_1)
(\Phi_{2n-1})\Big]\cdot m_1^*
\\
- \sum\limits_{m_1\in \binom{y,z}{n-1}}
\Big(p^{-1}\Big[(y\mu_{0,n-1})(\Phi_{2n-1}) \Big]\Big)\Big[(zm_1)(\Phi_{2n-1})\Big]\cdot m_1^*,
\end{cases}
\\
f_2(\nu_{0,n})={}&
\begin{cases}
\phantom{+} \sum\limits_{m_1\in \binom{y,z}{n-1}} \Big(p^{-1}\Big[(ym_1)(\Phi_{2n-1}) \Big]\Big)(z(\nu_{0,n}))\cdot m_1^*\\
- \sum\limits_{m_1\in \binom{y,z}{n-1}}\Big(p^{-1}\Big[(zm_1)(\Phi_{2n-1}) \Big]\Big)(y(\nu_{0,n}))\cdot m_1^*
,\end{cases}
\\
f_3(\nu_{0,n})={}&-y\t z(\nu_{0,n})+z\t y(\nu_{0,n}),
\\
f_4(\mu_{0,n-1}) ={}&\begin{cases}
\phantom{+}\sum\limits_{m_2\in \binom{y,z}{n}}
\Big(p^{-1}\Big[(z\mu_{0,n-1})(\Phi_{2n-1}) \Big]\Big)(y(m_2^*))
\cdot m_2
\\
-
\sum\limits_{m_2\in \binom{y,z}{n}}
 \Big(p^{-1}\Big[(y\mu_{0,n-1})(\Phi_{2n-1}) \Big]\Big)(z(m_2^*))
\cdot m_2, 
\end{cases}\\
f_5(\mu_{0,n-1})={}&y\t z\mu_{0,n-1} -z\t y\mu_{0,n-1},\text{ and} \\
f_6(\nu_{0,n})={}&\begin{cases}
\phantom{+}\sum\limits_{m_2\in \binom{y,z}{n}}
[p^{-1}(z(\nu_{0,n}))][y(m_2^*)]
\cdot m_2\\
- \sum\limits_{m_2\in \binom{y,z}{n}}
[p^{-1}(y(\nu_{0,n}))][z(m_2^*)]
\cdot m_2.
\end{cases}\end{align*}
\endgroup
The coefficients in $f_1$, $f_2$, $f_4$, and $f_6$ all are elements of $\kk$; so we have written 
$$\text{coefficient}\cdot \text{basis vector}.$$On the other hand, the coefficients of $f_3$ and $f_5$ are homogeneous of degree one in $R$; so we continue to write $$\text{coefficient}\t \text{basis vector}.$$
We complete the proof of Theorem~\ref{NewTheorem} by showing  
that the matrices of $f_1,\dots,f_6$, with respect to the ordered bases of \ref{order} are
$$\begin{array}{|c|c||c|c|}\hline
\text{map}&\text{matrix}&\text{map}&\text{matrix}\\\hline
f_1&A'&
f_2&B_1\fakeht\\\hline
f_3&B_2&
f_4&-B_1\transpose\fakeht\\\hline
f_5&-B_2\transpose&
f_6&D_0-D_0\transpose\fakeht\\\hline\end{array}\,.$$

The matrices $\bp$ and $\br$ are defined in the statement of Theorem~\ref{NewTheorem}. Notice that $\bp$ is the matrix for the map $p:\Sym_{n-1}U\to D_{n-1}U^*$ of Data~\ref{3.1} in the bases of \ref{order}. Let $$r:\Sym_nU_0\to D_{n-1}U^*$$ be the map defined by $$r(\mu_{0,n})=\mu_{0,n}(\Phi_{2n-1}),$$ for $\mu_{0,n}\in  \Sym_nU_0$. Observe that $\br$ is the matrix for $r$ in the bases of \ref{order}.

The module $\Sym_nU$ is equal to the direct sum $$x\Sym_{n-1}U\p \Sym_nU_0.$$ 
The choice of basis for $\Sym_nU$, as described in (\ref{favorite}), leads to the observation that  the matrix for $\Sym_nU\to D_{n-1}U^*$, given by
\begin{equation}\mu_n\mapsto \mu_n(\Phi_{2n-1}),\label{pre-pr}\end{equation} for $\mu_n\in \Sym_nU$, is 
\begin{equation}\label{pr}\left[\begin{array}{c|c}
\bp&\br\end{array}\right].\end{equation}
Furthermore, the dual of (\ref{pre-pr}) is the map \begin{equation}\label{**} \Sym_{n-1}U\to D_nU^*,\end{equation}which is given by $$\mu_{n-1}\mapsto \mu_{n-1}(\Phi_{2n-1}).$$ The matrix for (\ref{**}) is
$$\left[\begin{array}{c|c}\bp&\br\end{array}\right]\transpose=\bmatrix \fakeht \bp\transpose\\\hline\fakeht  \br\transpose
\endbmatrix.$$

\bs
Let $f_{11}$ and $f_{12}$ be the homomorphisms 
$$\Sym_{n-1}U_0\to D_{n-1}U_0^*$$
given by 
\begin{align*}f_{11}(\mu_{0,n-1})={}&\sum\limits_{m_1\in \binom{y,z}{n-1}}
\Big(p^{-1}\Big[(z\mu_{0,n-1})(\Phi_{2n-1}) \Big]\Big)\Big[(ym_1)
(\Phi_{2n-1})\Big]\cdot m_1^*\text{ and}\\
f_{12}(\mu_{0,n-1})={}&\sum\limits_{m_1\in \binom{y,z}{n-1}}
\Big(p^{-1}\Big[(y\mu_{0,n-1})(\Phi_{2n-1}) \Big]\Big)\Big[(zm_1)
(\Phi_{2n-1})\Big]\cdot m_1^*.\end{align*}
We first prove that $A_0$ is the matrix for $f_{11}$.
 Observe that
\begin{equation}\label{promise}f_{11}(\mu_{0,n-1})=y \Big(\proj_{D_nU_0^*} \big[\big(p^{-1} [(z\mu_{0,n-1})(\Phi_{2n-1})]\big)(\Phi_{2n-1})\big]\Big).\end{equation}(See Lemma~\ref{details} for details.) Thus, $f_{11}$ is the composition of the following six maps: 
\begin{align*}
\Sym_{n-1}U_0&\xrightarrow{\map_1}\Sym_nU_0
\xrightarrow{\map_2}D_{n-1}U^*
\xrightarrow{\map_3}\Sym_{n-1}U
\xrightarrow{\map_4}D_nU^*
\xrightarrow{\map_5}D_nU_0^*\\
&\xrightarrow{\map_6}D_{n-1}U_0^*, \end{align*}
where   \begin{align}
&\map_1(\mu_{0,n-1})=z\mu_{0,n-1},&&\map_2=r,&&\map_3=p^{-1},\notag\\
&\map_4(\mu_{n-1})=\mu_{n-1}(\Phi_{2n-1}),&&\map_5((x\mu_{n-1})^*)=0,\label{6543}\\ &\map_5((\mu_{0,n})^*)=(\mu_{0,n})^*,\text{ and}
&&\map_6(\nu_{0,n})=y(\nu_{0,n}),\notag\end{align}
for $$\mu_{0,i}\in \Sym_iU_0,\quad  \mu_{i}\in \Sym_iU, 
 \quad \text{and}\quad \nu_{0,i}\in D_iU_0^*.$$
The matrix for $\map_1$ in the bases of \ref{order} is
$$\begin{array}{|c||c|c|c|c|c|}\hline
&y^{n-1}&y^{n-2}z&\cdots&yz^{n-2}&z^{n-1}\\\hline\hline
y^n       &0&0&\cdots&0&0\\\hline
y^{n-1}z  &1&0&\cdots&0&0\\\hline
y^{n-2}z^2&0&1&\cdots&0&0\\\hline
\vdots    &\vdots&\ddots&\ddots&\ddots&\vdots\\\hline
yz^{n-1}  &0&0&\cdots&1&0\\\hline
z^n       &0&0&\cdots&0&1\\\hline
\end{array}\, ;$$the matrix for $\map_2$ is $\br$; and the matrix for $\map_3$ is $\bp^{-1}$. Observe that  
$\map_4$ is (\ref{**}); hence, the matrix for $\map_4$ 
is 
$$\bmatrix \bp\transpose\fakeht\\\hline \fakeht\br\transpose\endbmatrix.$$Use the convention 
of  \ref{order}
to see that the matrix for $\map_5$ is  $$\left[\begin{array}{c|c}
0_{(n+1)\times \binom{n+1}2}& I_{n+1} \\\end{array}\right]; $$and the matrix for $\map_6$ is 
$$\begin{array}{|c||c|c|c|c|c|c|}\hline
&(y^{n})^*&(y^{n-1}z)^*&\cdots&(y^2z^{n-2})^*&(yz^{n-1})^*&(z^n)^*\\\hline\hline
(y^{n-1})^*       &1&0&\cdots&0&0&0\\\hline
(y^{n-2}z)^*  &0&1&\ddots&0&0&0\\\hline
\vdots    &\vdots&\ddots&\ddots&\ddots&\vdots&\vdots\\\hline
(yz^{n-2})^*&0&0&\ddots&1&0&0\\\hline
(z^{n-1})^*  &0&0&\cdots&0&1&0\\\hline
\end{array}.$$
The matrix for $\map_5\circ \map_4\circ\map_3\circ\map_2$ is easily seen to be $\br\transpose\bp^{-1}\br$. The matrix for $\map_1$ deletes column $1$ and the matrix for $\map_6$ deletes row $n+1$. 
It follows that $A_0$ is the matrix for $f_{11}$.

In a similar manner the matrix for $f_{12}$ is 
$$A_0'=[\br\transpose \bp^{-1}\br]
\text{ with row $1$ and
column $n+1$ deleted}.$$ The matrix $\bp$ is symmetric; hence, the matrix $[\br\transpose \bp^{-1}\br]$ is symmetric and $A_0'$ is equal to the transpose of $A_0$. We conclude that the matrix for $f_1$, which is equal to $f_{11}-f_{12}$, is $A_0-A_0\transpose=A'$.

\bs

Let $f_{21}$ and $f_{22}$ be the homomorphisms 
$$D_nU_0^*\to D_{n-1}U_0^*$$ given by
\begin{align*}f_{21}(\nu_{0,n})={}& \sum\limits_{m_1\in \binom{y,z}{n-1}} \Big(p^{-1}\Big[(ym_1)(\Phi_{2n-1}) \Big]\Big)(z(\nu_{0,n}))\cdot m_1^*\text{ 
and} \\
f_{22}(\nu_{0,n})={}&\sum\limits_{m_1\in \binom{y,z}{n-1}}\Big(p^{-1}\Big[(zm_1)(\Phi_{2n-1}) \Big]\Big)(y(\nu_{0,n}))\cdot m_1^*.\end{align*}
Use the techniques of Lemma~\ref{details} to see that 
$$f_{21}(\nu_{0,n})=y \big[ \proj_{D_nU_0^*}
 \big([p^{-1}(z(\nu_{0,n}))](\Phi_{2n-1})\big)\big].
$$
Thus, the map $f_{21}$ is the composition of the following six maps:
\begin{align*}D_nU_0^*&\xrightarrow{\map_7}D_{n-1}U_0^*\xrightarrow{\map_8}D_{n-1}U^*\xrightarrow{\map_3}\Sym_{n-1}U\xrightarrow{\map_4}D_nU^*\xrightarrow{\map_5}D_nU_0^*\\&\xrightarrow{\map_6}D_{n-1}U_0^*,\end{align*}where 
$\map_7(\nu_{n,0})=z(\nu_{n,0})$, $\map_8$ is inclusion, 
and $\map_3$, $\map_4$, $\map_5$, and $\map_6$ are given defined in (\ref{6543}).  
The matrix for $\map_7$ in the bases of \ref{order} is
\begin{align}&\begin{array}{|c|c|c|c|c|c|c|}\hline
&(y^n)^*&(y^{n-1}z)^*&(y^{n-2}z)^*&\dots&(yz^{n-1})^*&(z^n)^*\\\hline
(y^{n-1})^*&0&1&0&0&\cdots&0\\\hline
(y^{n-2}z)^*&0&0&1&0&\cdots&0\\\hline
\vdots&\vdots&\vdots&0&\ddots&\ddots&\vdots\\\hline
(yz^{n-2})^*&0&\vdots&\vdots&\ddots&1&0\\\hline
(z^{n-1})^*&0&0&0&\cdots&0&1\\\hline\end{array}\notag\\
\label{map_7}={}&\left[\begin{array}{c|c} 0_{n\times 1}&I_n\end{array}\right];\end{align}
the matrix for $\map_8$ is
\begin{equation}\label{map_8}\bmatrix 0_{\binom n2\times n}\\\hline I_{n}\endbmatrix;\end{equation}
the matrix for $\map_6\circ \map_5\circ \map_4\circ \map_3$ continues to be
$$\left[\begin{array}{c|c} I_n&0_{n\times 1}\end{array}\right] \br\transpose\bp^{-1}.$$ Observe that 
$$\br\transpose\bp^{-1}\bmatrix 0_{\binom n2\times n}\\\hline I_{n}\endbmatrix
=[ \text{the right most $n$ columns of $\br\transpose\bp^{-1}$}]=B_0.$$
It follows that  the matrix for $f_{21}$
is $$\left[\begin{array}{c|c} I_n&0_{n\times 1}\end{array}\right]B_0
\left[\begin{array}{c|c} 0_{n\times 1}&I_n\end{array}\right]=
\left[\begin{array}{c|c} 0_{n\times 1}& B_0\text{ with row $n+1$ deleted}\end{array}\right].$$
The map $f_{22}$ is obtained from the map $f_{21}$ by exchanging the role of $y$ and $z$; hence the matrix for $f_{22}$ is
$$
\left[\begin{array}{c|c} 0_{n\times 1}&I_n\end{array}\right]B_0
\left[\begin{array}{c|c} I_n&0_{n\times 1}\end{array}\right]
=
\left[\begin{array}{c|c} B_0\text{ with row $1$ deleted}
& 0_{n\times 1}
\end{array}\right].$$
It follows that the matrix for $f_2=f_{21}-f_{22}$ is
$$\left[\begin{array}{c|c} 0_{n\times 1}& B_0\text{ with row $n+1$ deleted}\end{array}\right]-
\left[\begin{array}{c|c} B_0\text{ with row $1$ deleted}
& 0_{n\times 1}
\end{array}\right]=B_1.$$

\bs 
We have already seen that the matrix for $\map_7(\nu_{0,n})=z(\nu_{0,n})$ is
$\left[\begin{array}{c|c} 0_{n\times 1}&I_n\end{array}\right]$ and the matrix for
$\map_6(\nu_{0,n})=y(\nu_{0,n})$ is $\left[\begin{array}{c|c}I_n&0_{n\times 1}\end{array}\right]$. It follows that the matrix for $f_3$ is  
$$\left[\begin{array}{c|c}zI_n&0_{n\times 1}\end{array}\right]
-\left[\begin{array}{c|c} 0_{n\times 1}&yI_n\end{array}\right]=B_2.$$

\bs The maps $f_2$ and $f_4$ satisfy
$$[f_2(\nu_{0,n})](\mu_{0,n-1})+\nu_{0,n}[f_4(\mu_{0,n-1})]=0;$$ hence the matrix for $f_4$ is minus the transpose of the matrix for $f_2$. Similarly,
$$[f_3(\nu_{0,n})](\mu_{0,n-1})+\nu_{0,n}[f_5(\mu_{0,n-1})]=0;$$ hence the matrix for $f_5$ is minus the transpose of the matrix for $f_3$.

\bs
Let $f_{61}$ and $f_{62}$ be the homomorphisms 
$$D_nU_0^*\to \Sym_nU_0,$$ which are given by
\begin{align*}
f_{61}(\nu_{0,n})= {}&\sum\limits_{m_2\in \binom{y,z}{n}}
[p^{-1}(z(\nu_{0,n}))][y(m_2^*)]
\cdot m_2\text{ and}\\
f_{62}(\nu_{0,n})={}&
\sum\limits_{m_2\in \binom{y,z}{n}}
[p^{-1}(y(\nu_{0,n}))][z(m_2^*)]
\cdot m_2.\end{align*} 
Use the techniques of Lemma~\ref{details} to see that  $$f_{61}(\nu_{0,n})=y\cdot \proj_{\Sym_{n-1}U_0}[p^{-1}(z(\nu_{0,n}))].$$ 
The map $f_{61}$  
is the composition
$$D_n(U_0^*)\xrightarrow{\map_7}D_{n-1}U_0^*\xrightarrow{\map_3\circ\map_8}\Sym_{n-1}U\xrightarrow{\map_9}\Sym_{n-1}U_0\xrightarrow{\map_{10}}\Sym_nU_0^,$$
where  
$\map_9$ is projection, and $\map_{10}$ is multiplication in $\Sym_\bullet U_0$ by $y$.
The matrix for $\map_7$ is given in (\ref{map_7}); the matrix for $\map_3\circ \map_8$ is the product of $\bp^{-1}$ and (\ref{map_8});
the matrix for $\map_9$ is 
$$\left[\begin{array}{c|c} 0_{n\times \binom n2}&I_n\end{array}\right];$$ and the matrix for $\map_{10}$ is $$\bmatrix I_n\\\hline 0_{1\times n}\endbmatrix.$$
It follows that the matrix for $f_{61}$ is
the product
\begingroup\allowdisplaybreaks
\begin{align*}
&\bmatrix I_n\\\hline 0_{1\times n}\endbmatrix \left[ \begin{array}{c|c} 0_{n\times \binom n2}&I_n\end{array}\right]\bmatrix\text{The right most $n$ columns}\\\text{of $\bp^{-1}$}.\endbmatrix \left[ \begin{array}{c|c} 0_{n\times 1}&I_n\end{array}\right]\\
{}={}&\bmatrix I_n\\\hline 0_{1\times n}\endbmatrix 
\bmatrix \text{the bottom right $n\times n$ submatrix of $\bp^{-1}$}\endbmatrix  
\left[ \begin{array}{c|c} 0_{n\times 1}&I_n\end{array}\right]\\
{}={}&\left[ \begin{array}{c|c} 0_{n\times 1}&\text{the bottom right $n\times n$ submatrix of $\bp^{-1}$}\\\hline
0_{1\times 1}&0_{1\times n}\end{array}
\right]=D_0.\end{align*}\endgroup
The matrix for $f_{62}$ is obtained from the matrix for $f_{61}$ by exchanging the role of 
$y$ and $z$. In other words, the matrix for $f_{62}$ is  
\begin{align}\notag&\left[ \begin{array}{c}0_{1\times n}\\\hline I_n\end{array}\right]
\bmatrix \text{the bottom right $n\times n$ submatrix of $\bp^{-1}$}\endbmatrix
\left[ \begin{array}{c|c}I_n&0_{n\times 1}\end{array}\right]\\
\label{6.11.8}{}={}&\left[ \begin{array}{c|c} 0_{1\times n}&0_{1\times 1}\\\hline
\text{the bottom right $n\times n$ submatrix of $\bp^{-1}$}&0_{n\times 1}
\end{array}\right].\end{align}
The matrix $\bp$ is symmetric; so (\ref{6.11.8}) is the transpose of the matrix for $f_{61}$ and the matrix for $f_6=f_{61}-f_{62}$ is $D_0-D_0\transpose$.
\hfill \qed
\end{chunk}
At  (\ref{promise}) we promised to exhibit the following calculation.  
\begin{lemma}
\label{details} The elements 
$$\sum\limits_{m_1\in \binom{y,z}{n-1}}
(p^{-1}[(z\mu_{0,n-1})(\Phi_{2n-1}) ])[(ym_1)
(\Phi_{2n-1})]\cdot m_1^*$$ and 
$$y \big(\proj_{D_nU_0^*} \big[\big(p^{-1} [(z\mu_{0,n-1})(\Phi_{2n-1})]\big)(\Phi_{2n-1})\big]\big)$$ of $D_{n-1}U_0^*$ are equal.
\end{lemma}
\begin{proof}
 Recall that $p^{-1}[(z\mu_{0,n-1})(\Phi_{2n-1})]$ and $y$ and $m_1$ are all elements of the commutative ring $\Sym_\bullet U$ and $\Phi_{2n-1}$ is an element of the $\Sym_\bullet U$-module $D_\bullet U^*$. It follows that
\begin{align*}&\big(p^{-1}[(z\mu_{0,n-1})(\Phi_{2n-1})]\big)[(ym_1)
(\Phi_{2n-1})]\\
{}={}&m_1
\big[y 
\big((p^{-1}[(z\mu_{0,n-1})(\Phi_{2n-1})])(\Phi_{2n-1})\big)\big]
,
\end{align*}
where 
\begin{align*}(p^{-1}[(z\mu_{0,n-1})(\Phi_{2n-1})])&(\text{an element of $D_\bullet U^*$}),\\
y&(\text{an element of $D_\bullet U^*$}),\text{ and}\\
m_1&(\text{an element of $D_\bullet U^*$})
\end{align*}  
all represent the module action of $\Sym_\bullet U$ on $D_\bullet U^*$.
If $X$ is an element of $D_{n-1}U^*$, then 
$$\sum\limits_{m_1\in \binom{y,z}{n-1}}m_1(X)\cdot m_1^*=\proj_{D_{n-1}U_0^*}X$$
because
$$\sum\limits_{m_1\in \binom{y,z}{n-1}}m_1(X_0)\cdot m_1^*=X_0,$$if $X_0\in D_{n-1}U_0^*$; and
$$\sum\limits_{m_1\in \binom{y,z}{n-1}}m_1(X_1)\cdot m_1^*=0,$$if $X_1$ is in the submodule of $D_{n-1}U$ which is generated by$$\left\{(xm)^*\left| m\in\binom{x,y,z}{n-2}\right.\right\}.$$
Thus, \begingroup\allowdisplaybreaks\begin{align*}&\sum\limits_{m_1\in \binom{y,z}{n-1}}
(p^{-1}[(z\mu_{0,n-1})(\Phi_{2n-1}) ])[(ym_1)
(\Phi_{2n-1})]\cdot m_1^*\\{}={}&\proj_{D_{n-1}U_0^*}
\big(y \big[(p^{-1}[(z\mu_{0,n-1})(\Phi_{2n-1})])(\Phi_{2n-1})\big]\big)\\
{}={}&y\big(\proj_{D_{n}U_0^*}
\big[(p^{-1}[(z\mu_{0,n-1})(\Phi_{2n-1})])(\Phi_{2n-1})\big]\big)
.\end{align*}\endgroup The final equality holds because  
the diagram
$$\xymatrix{D_nU^*\ar[r]^{\proj}\ar[d]^{y}&D_nU_0^*\ar[d]^{y}\\D_{n-1}U^*\ar[r]^{\proj}&D_{n-1}U_0^*}$$ commutes.
\end{proof}
\begin{remark} 
We proved Theorem~\ref{NewTheorem} by proving that the $b_2$ of Theorem~\ref{NewTheorem} is the matrix for the coordinate-free map $b_2$ of Theorem~\ref{3.3}. In each of the Theorems, \ref{NewTheorem} and \ref{3.3}, the image of $b_3$ is the kernel of $b_2$ and $b_1=\Hom_R(b_3,R)$. It follows that $b_1$ from Theorem~\ref{3.3} is equal to a unit of $\kk$ times $b_1$ from Theorem~\ref{NewTheorem}. There are times that  the explicit formulas of Theorem~\ref{3.3} are more useful  than the instruction ``take the maximal order Pfaffians of this alternating matrix''.\end{remark}

\section{The main theorem.}
Theorem~\ref{main} is the main result in the paper.
\begin{observation}\label{O4.1}
Let $R=\kk[x,y,z]$ be a standard graded polynomial ring over a field. If $J$ is a quadratically presented grade three Gorenstein ideal in $R$, then the Betti numbers of the minimal free resolution of $R/J$ by free $R$-modules have the form of {\rm(\ref{BN})} for some positive even number $n$.
In particular, the socle degree of $R/J$ is $2n-1$ and the minimal homogeneous generators of $J$ all have degree $n$. 
\end{observation}

\begin{proof} 
Apply the Theorem of Buchsbaum and Eisenbud \cite{BE77} to see that $J$ is minimally generated by the maximal order Pfaffians of some alternating matrix $X$ of odd size, whose entries, by our hypothesis, are quadratic forms in $R$. Let $X$ have $n+1$ rows and columns, where $n$ is even; and let $\underline{x}$  be the row vector of signed maximal order Pfaffians of $X$. (See \ref{2.7}, if necessary.) The minimal homogeneous resolution of $R/J$ by free $R$-modules has the form
$$0\to R(-(2n-2))\xrightarrow{ \underline{x}\transpose}R(-(n+2))^{n+1}
\xrightarrow{X} R(-n)^{n+1} \xrightarrow{ \underline{x}} R.$$\end{proof}

In Theorem~\ref{main}, $J$ is a homogeneous grade three Gorenstein ideal $R=\kk[x,y,z]$.  We assume that the socle degree of $R/J$ is $2n-1$, for some positive even integer $n$, and that all homogeneous generators of $J$ have
degree at least $n$. We also assume that $x$ is a weak Lefschetz element on $R/J$. We determine if $J$ is quadratically presented; and if so, then we give the presentation matrix of $J$ in terms of the Macaulay inverse system for $J$. 

\begin{theorem}
\label{main}
Let $R=\kk[x,y,z]$ be a standard graded polynomial ring over the  field $\kk$ and $n$ be a positive even integer. Let $U$ be the vector space $R_1$, $\Phi_{2n-1}$ be an element of $D_{2n-1}U^*$ and $J$ be the ideal $$J=\ann_{\Sym_\bullet U} \Phi_{2n-1}$$ of $R=\Sym_\bullet U$. Assume that $J_{n-1}=0$ and that $x$ is a weak Lefschetz element on $R/J$. Let
$$I=\ann_{\Sym_\bullet U} x(\Phi_{2n-1}).$$
Then the following statements hold.
\begin{enumerate}[\rm(a)]
\item\label{main.a} The ideal $I$ is a linearly presented grade three Gorenstein ideal in $R$ and the minimal homogeneous resolution of $R/I$ by free $R$-modules is 
$$0\to R(-2n-1)\xrightarrow{b_3} R^{2n+1}(-n-1)\xrightarrow{b_2}R^{2n+1}(-n)\xrightarrow{b_1}R,$$
where 
$$b_2=\left[\begin{array}{c|c} xA'& B\\\hline\fakeht
-B\transpose&D\end{array}\right],$$
 as described in Theorem~{\rm \ref{NewTheorem}}.
\item\label{main.b} Let $g_i$ be $(-1)^{i+1}$ times the Pfaffian of the submatrix of $b_2$ obtained by deleting row and column $i$. Then $g_{n+1},\dots,g_{2n+1}$ is part of a minimal generating set for $J$.
\item\label{main.c} The ideal $J$ is quadratically presented if and only if 
the matrix $A'$ is invertible.
\item\label{main.d} If $J$ is quadratically presented, then $J=(g_{n+1},\dots,g_{2n+1})$ and 
the minimal homogeneous resolution of $R/J$ by free $R$-modules is
\begin{equation}
\label{resol}0\to R(-2n-2)\xrightarrow{c_3} R(-n-2)^{n+1}\xrightarrow{c_2}
R(-n)^{n+1}\xrightarrow{c_1}R,\end{equation} where
$$c_2 =B\transpose (A')^{-1}B+xD,$$
$c_1 $ is equal to the row vector of maximal order Pfaffians of $ c_2$, and $c_3 =c_1 \transpose$.
\end{enumerate}
\end{theorem}

\begin{proof}
(\ref{main.a}) Let $p:\Sym_{n-1}U\to D_{n-1}U^*$ be the $\kk$-module homomorphism with
$$p(\mu_{n-1})=\mu_{n-1}(x\Phi_{2n-1}), \quad\text{for $\mu_{n-1}\in \Sym_{n-1}U$}.$$
According to Proposition~\ref{1.8} it suffices to show that $p$ is injective. Suppose 
  $\mu_{n-1}$, in $\Sym_{n-1}U$, is  in the kernel of $p$. Then $(x\mu_{n-1})(\Phi_{2n-1})=0$; $x\mu_{n-1}\in J$; and $\mu_{n-1}$ represents an element in the kernel of the $\kk$-module map
\begin{equation}\label{full rank}(R/J)_{n-1}\xrightarrow{\ \ x\ \ }(R/J)_n,\end{equation} which is given by multiplication by $x$.
On the other hand, the hypothesis that $x$ is a weak Lefschetz element on $R/J$ guarantees that the rank of (\ref{full rank}) is equal to
$$\min \{\dim_\kk (R/J)_{n-1}, \dim_\kk(R/J)_n\}.$$ The Hilbert function of the standard graded Artinian Gorenstein algebra $R/J$, with socle degree $2n-1$, is symmetric in the sense that $$\dim_\kk (R/J)_{i}=\dim_\kk(R/J)_{2n-1-i},$$ for all $i$. It follows that $(R/J)_{n-1}$ and $(R/J)_n$ have the same dimension;
(\ref{full rank}) is injective; and $\mu_{n-1}\in J_{n-1}=0$.

\ms\noindent(\ref{main.b}) Recall that Theorem~\ref{NewTheorem} is Theorem~\ref{3.3} written in the language of matrices rather than the language of multilinear algebra. The assertion of (\ref{main.b}) is obvious in the language of multilinear algebra. Indeed, the row vector
\begin{equation}\label{late gs}\bmatrix g_{n+1},&\dots,&g_{2n+1}\endbmatrix\end{equation}
is equal to some unit of $\kk$ times the row vector
$$\bmatrix b_1(m_1),&\dots,&b_1(m_{n+1})\endbmatrix,$$
where $m_1,\dots, m_{n+1}$ is the basis $y^n,y^{n-1}z,\dots,yz^{n-1},z^n$ of $\Sym_nU_0$, as given in \ref{order}, and 
$$b_1(\mu_{0,n})=\mu_{0,n}-xp^{-1}(\mu_{0,n}(\Phi_{2n-1})),\quad\text{for $\mu_{0,n}\in \Sym_nU_0$,}$$
 as defined in the statement of Theorem~\ref{3.3}. Fix 
$\mu_{0,n}\in \Sym_{n}U_0$. 
 We show that $b_1(\mu_{0,n})$ is in $J$ by showing that $[b_1(\mu_{0,n})](\Phi_{2n-1})=0$. Observe that
\begin{align*}[xp^{-1}(\mu_{0,n}(\Phi_{2n-1}))](\Phi_{2n-1})
{}={}&[p^{-1}(\mu_{0,n}(\Phi_{2n-1}))](x(\Phi_{2n-1}))\\
{}={}&p[p^{-1}(\mu_{0,n}(\Phi_{2n-1}))]
{}={}\mu_{0,n}(\Phi_{2n-1}).\end{align*} 
Hence, $g_{n+1},\dots,g_{2n+1}$ are elements of $J$ of degree $n$ which form part of a minimal generating set of $I$. By hypothesis, there are no elements of $J$ of degree less than $n$. It follows that $g_{n+1},\dots,g_{2n+1}$ is part of a minimal generating set of $J$. Let $J'$ be the ideal $(g_{n+1},\dots,g_{2n+1})$. 

\ms\noindent(\ref{main.c}) and (\ref{main.d}) We first show that if $A'$ is not invertible, then there is a non-zero linear relation on the minimal generators of $J$. 

Recall that $A'$ is an $n\times n$ matrix of elements of $\kk$. If $A'$ is not invertible, then there exists a non-zero vector $v\in \kk^n$ with $Av\neq 0$.
The matrix $b_2$ appears in a minimal homogeneous resolution; so, the columns of $b_2$ are linearly independent and  $$b_2\bmatrix v\\\hline 0\endbmatrix =\bmatrix 0\\\hline \fakeht-B\transpose v\endbmatrix$$ is a non-zero linear relation on $b_1$. It follows that $-B\transpose v$ is a non-zero linear relation on (\ref{late gs}). However, according to (\ref{main.b}), $g_{n+1},\dots,g_{2n+1}$ is part of a minimal generating set for $J$. We conclude that if $A'$ is not invertible, then there is a non-zero linear relation on a minimal generating set of $J$ and $J$ is not quadratically presented. 

Now we assume that $A'$ is invertible. We prove that $J$ is quadratically presented and that the complex of (\ref{main.d}) is the minimal homogeneous  resolution of $R/J$ by free $R$-modules.

\begin{claim} 
\label{Claim1}
If $A'$ is invertible, then the row vector {\rm(\ref{late gs})} is a unit of $\kk$ times the row vector of  
signed maximal order Pfaffians of $c_2=B\transpose (A')^{-1}B+xD$.\end{claim}

Fix the index $i$ with $1\le i\le n+1$; let
 $B'$ be $B$ with column $i$ deleted and $D'$ be $D$ with row $i$ and column $i$ deleted. 
Let $\Pf$ mean ``the Pfaffian of''. The element $x$ of the polynomial ring $R$ is a regular element. No damage is done if we compute in the ring $R[\frac 1x]$.
Observe that the Pfaffian of $b_2$ with row and column $n+i$ deleted is equal to 
\begingroup\allowdisplaybreaks
\begin{align*}
&{}\Pf\left[\begin{array}{c|c} xA'&B'\\\hline-(B')\transpose&D'\fakeht\end{array}\right]
\\ 
{}={}&\Pf \left(\left[\begin{array}{c|c} I_n&0\\\hline \frac1x (B'){}\transpose (A'){}^{-1}&I_n\fakeht\end{array}\right]   \left[\begin{array}{c|c} xA'&B'\\\hline-(B'){}\transpose&D'\fakeht\end{array}\right]
\left[\begin{array}{c|c} I_n&-\frac 1x(A'){}^{-1}B'
\vphantom{\Big)}\\\hline0&I_n\fakeht\end{array}\right]\right)
\\
{}={}&\Pf \left[\begin{array}{c|c} xA'&0\\\hline0&\frac 1x\left[(B'){}\transpose (A'){}^{-1}B'\right]+D'
\fakeht\end{array}\right]\\
{}={}&\Pf (xA')\cdot \Pf({\ts\frac 1x}\left[(B'){}\transpose (A'){}^{-1}B'+xD'\right])\\
{}={}&x^{n/2} \Pf(A'){\ts\frac 1{x^{n/2}}}\Pf\left[(B'){}\transpose (A'){}^{-1}B'+xD'\right],
\end{align*} \endgroup which is equal to the Pfaffian of $A'$ times the Pfaffian of $c_2$ with row and column $i$ deleted. This completes the proof of Claim~\ref{Claim1}.

\begin{claim}
\label{Claim2} The ideal $J'$ is a quadratically presented grade three Gorenstein ideal and the minimal homogeneous resolution of $R/J'$ is 
given in {\rm(\ref{resol})}.\end{claim}

The row vector (\ref{late gs}) is the row vector of signed maximal order Pfaffians of the $(n+1)\times(n+1)$ alternating matrix $c_2$, and $n+1$ is odd. We show that 
$$0\to R(-2n-2)\xrightarrow{(\text{\rm\ref{late gs}})\transpose} 
R(-n-2)^{n+1}\xrightarrow{c_2}R(-n)^{n+1}\xrightarrow{(\text{\rm\ref{late gs}})} R$$ is the minimal homogeneous resolution of $R/J'$ by free $R$-modules by showing that the grade of $J'$ is at least three; see \cite{BE77} or \cite{BE73}. In particular, we show that $(x,y,z)$ is contained in the radical of $J'$. The socle degree of $R/I$ is $2n-2$; so the ideal $(x,y,z)^{2n-1}$ is contained in $I$.

The fact that the columns of
$$\bmatrix xA'\\\hline\fakeht -B\transpose\endbmatrix$$ are relations on $[g_1,\dots,g_{2n+1}]$, with $A'$ invertible, $(g_{n+1},\dots,g_{2n+1})=J'$, and $I$ equal to $(g_1,\dots,g_n)+J'$ ensures that $xI\subseteq J'$. Thus, $x^{2n}\in J'$. 

The matrix $B_2$ is the $n\times (n+1)$ matrix
$$B_2=\left[\begin{array}{c|c|c|c|c|c}
z&-y&0&\cdots&0&0\\\hline
0&z&-y&\ddots&\vdots&\vdots\\\hline
\vdots&\ddots&\ddots&\ddots&0&\vdots\\\hline
\vdots&\ddots&\ddots&\ddots&-y&0\\\hline
0&\cdots&\cdots&0&z&-y\end{array}\right].$$ 
The determinant of $B_2$ with column $n+1$ deleted is $z^n$ and the determinant of $B_2$ with column $1$ deleted is $y^n$. It is well known that if $A'$ is an alternating $n\times n$ matrix and $M$ is an $n\times n$ matrix, then the Pfaffian of $M\transpose A' M$ is $\det M \Pf A'$; see for example \cite[Thm. 3.28]{A57}. It follows that $z^n$ and $y^n$ are in the ideal generated by the maximal order Pfaffians of $c_2$, together with $x$.  We conclude that $(x,y,z)$ is contained in the radical of $J'$.
This completes the proof of Claim~\ref{Claim2}.

\bs We complete the proof by showing that $J'=J$. For this we use the Macaulay duality. The ideals under consideration are related by
\begin{equation}\label{incl} J'\subseteq J\subseteq I.\end{equation}
The Macaulay inverse systems
$$J=\ann_{\Sym_\bullet}\Phi_{2n-1} \quad\text{and}\quad I=\ann_{\Sym_\bullet}x(\Phi_{2n-1})$$ are known. The socle degree of $R/J'$ is $2n-1$. Let $\Psi_{2n-1}\in D_{2n-1}U^*$ be the Macaulay inverse system of $J'$; hence $J'=\ann _{\Sym_\bullet}\Psi_{2n-1}$. The inclusions of (\ref{incl}) give rise to
inclusions 
$$\ann_{D\bullet U^*}I\subseteq \ann_{D\bullet U^*}J\subseteq \ann_{D\bullet U^*}J'.$$If $\Phi$ is in $D\bullet U^*$, then 
$$\ann_{D_\bullet U^*}(\ann_{\Sym_\bullet U}(\Phi))=(\Phi),$$ where $(\Phi)$ is the $\Sym_\bullet U$-submodule of $D_\bullet U^*$ generated by $\Phi$. It follows that 
$$(x\Phi_{2n-1})\subseteq (\Phi_{2n-1})\subseteq (\Psi_{2n-1}).$$Thus,
$\Phi_{2n-1}=\alpha \Psi_{2n-1}$ for some unit $\alpha$ in $\kk$ and $J=J'$.
\end{proof}

\section{An example.}\label{sect5}
\begin{proposition}
\label{5.1}If $\kk$ is a field of  characteristic  zero, $R=\kk[x,y,z]$, 
 $n$ is an even integer, and
$$J=(x^{n+1},y^{n+1},z^{n+1}):(x+y+z)^{n+1},$$ 
then the graded Betti numbers of $R/J$ are given in {\rm(\ref{BN})}
and $x$ is a weak Lefschetz element for
 $R/J$. 
\end{proposition}

\begin{proof}
The field $\kk$ has characteristic zero. Apply \cite[Obs.~3.13.(d)]{KRV22} (which is based on \cite[Thm.~5]{RRR91}) to see that $R/J$ is a compressed ring with socle degree $s=2n-1$. Now apply  \cite[Lem.~5.4]{KRV22} (which is based on \cite[Prop.~3.2]{B99}) to see that the minimal homogeneous resolution of $R/J$ by free $R$-modules has the form 
\begin{equation}0\to R(-2n-2)\to \begin{matrix} R(-n-1)^v\\\p\\R(-n-2)^{n+1}\end{matrix} \to\begin{matrix}
R(-n)^{n+1}\\\p\\R(-n-1)^{v}\end{matrix}
\to R,\label{1-28}\end{equation} 
for some integer $v$.
A minimal generating set for $J$ which consists of $n+1$ homogeneous forms of degree $n$ is given in \cite[Prop.~5.7]{KRV22} (with $\epsilon=0$ and $d=n+1$). In particular, $J$ has $v=0$ minimal generators of degree $n+1$ and (\ref{1-28}) is equal to (\ref{BN}).

We complete the argument by showing that $x$ is a Lefschetz element for $R/J$. We have seen that the minimal generators of $J$ all have degree $n$ and that the socle of $R/J$ has degree $s=2n-1$. In particular, $J_{(s-1)/2}=J_{n-1}=0$ and $(R/J)_{(s-1)/2}=R_{n-1}$. Let $N=\binom{n+1}2$, let $\mu_1,\dots,\mu_N$ be the monomials of $R$ of degree $n-1$, 
and  $\Phi$ in $D_{2n-1}R_1^*$ be a Macaulay inverse system for $J$, where $$(-)^*\text{ means }\Hom_{\kk}(-,\kk).$$
Define $\psi:(R/J)_s\to \kk$ by setting 
$$\psi(\bar \mu)=\mu(\Phi)$$ 
where $\mu\in R_s$ and $\bar \mu$ is the image of $\mu$ in $(R/J)_s$. Observe that $\psi:(R/J)_s\to \kk$ is an isomorphism.
 Apply Lemma~\ref{1-30}. In order to show that $x$ is a Lefschetz element for $R/J$,  it suffices to show that $\det(\psi(\mu_ix\mu_j))$ is non-zero. 

Observe that $$(\mu_i x \mu_j)(\Phi)=(\mu_i  \mu_j)(x(\Phi))$$ and that $x(\Phi)$ is a Macaulay inverse system for $$\mathcal J=(x^n,y^{n+1},z^{n+1}):(x+y+z)^{n+1}.$$ Apply \cite[Thm. 5]{RRR91} to see that $\mathcal J_{n-1}=0$. Observe that the socle degree of $R/\mathcal J$ is the socle degree of $R/(x^n,y^{n+1},z^{n+1})$ minus the degree of $(x+y+z)^{n+1}$; hence $2n-2$. It follows (see, for example, Proposition~\ref{1.8}) that $\mathcal J$ is 
 linearly presented; 
hence, in particular, the matrix $(\mu_i\mu_j(x\Phi))$ is invertible; see Proposition~\ref{1.8}. The proof is complete.   
\end{proof}

\begin{lemma} 
\label{1-30}
Let $R$ be a standard graded polynomial ring over a field $\kk$, and $J$ be a homogeneous ideal of $R$ with $R/J$ Gorenstein and Artinian with odd socle degree $s$. Fix an isomorphism $\psi:(R/J)_s\to \kk$. Let $\ell$ be a linear form in $R$ and $\mu_1,\dots,\mu_N$ be 
a basis for $(R/J)_{(s-1)/2}$. Consider the $N\times N$ matrix $M$ with entry $\psi(\mu_i\ell \mu_j)$ in row $i$ and column $j$. 
 Then  $\ell$ is a Lefschetz element for $R/J$ if and only if  
  $\det M\ne 0$.
 \end{lemma}
\begin{proof} $(\Rightarrow)$
Assume that $\ell$ is a Lefschetz element for $R/J$; thus, multiplication by $\ell$ from $(R/J)_{i}$ to $(R/J)_{i+1}$  is an injection whenever $\dim(R/J)_{i}\le \dim (R/J)_{i+1}$.    
The vector spaces $(R/J)_{(s-1)/2}$ and $(R/J)_{(s+1)/2}$ have the same dimension because $R/J$ is Gorenstein.
It follows that the multiplication    $$\ell:(R/J)_{(s-1)/2}\to (R/J)_{(s+1)/2}$$ is an isomorphism. On the other hand, the map
$$(R/J)_{(s-1)/2}\t (R/J)_{(s+1)/2}\to \kk,$$ 
which sends $\theta\t \theta'$ to $\psi(\theta\theta')$ is a perfect pairing, again because $R/J$ is Gorenstein. Thus, there exists a basis $\mu_1',\dots,\mu_N'$ for $(R/J)_{(s+1)/2}$ for which the matrix 
$$\psi(\mu_i\mu_j')$$ looks like the identity matrix. 
The  elements $\mu_1',\dots,\mu_N'$ and the elements $\ell\mu_1,\dots,\ell\mu_N$ each form a  basis for $(R/J)_{(s+1)/2}$; hence, the matrix $M$ is obtained from the identity matrix by a change of basis. We conclude that $\det M$ is also a unit.

\ms $(\Leftarrow)$
If $\det M$ is a unit, then the elements  $\ell\mu_1,\dots,\ell\mu_N$ of $(R/J)_{(s+1)/2}$ are linearly independent. Thus, multiplication by $\ell$ from 
$(R/J)_{(s-1)/2}$ to $(R/J)_{(s+1)/2}$ is injective and this is enough to show that $\ell$ is a Lefschetz element; see, for example, \cite[Prop.~2.1]{MMN11}.  
\end{proof}  

\begin{corollary}
\label{cor}If $\kk$ is a field of  characteristic  zero, $R=\kk[x,y,z]$, 
 $n$ is an even integer, and
\begin{equation}\label{J}J=(x^{n+1},y^{n+1},z^{n+1}):(x+y+z)^{n+1},\end{equation} 
then the minimal homogeneous resolution of $R/J$ is given by
\begin{equation}
\notag
0\to R(-2n-2)\xrightarrow{c_3} R(-n-2)^{n+1}\xrightarrow{c_2}
R(-n)^{n+1}\xrightarrow{c_1}R,\end{equation} where
$$c_2 =B\transpose (A')^{-1}B+xD,$$
$c_1 $ is equal to the row vector of maximal order Pfaffians of $ c_2$, and $c_3 =c_1 \transpose$, as described in Theorem~{\rm\ref{main}}.
\end{corollary}
\begin{proof}
Apply Proposition~\ref{5.1} to see that the hypotheses of Theorem~\ref{main} are satisfied. The element
$$\Phi_{2n-1}=(x+y+z)^{n+1}((x^ny^nz^n)^*)$$ of $D_{2n-1}R_1^*$ is a Macaulay inverse system for $J$.
\end{proof}
\begin{example}
 If $n=2$, then the matrices of Corollary~\ref{cor} are
$$\bp=\bmatrix
0 &3& 3 \\
 3& 3& 6 \\
3 &6& 3 \endbmatrix,\quad
\bp^{-1}=\frac 16\bmatrix
 -3& 1 & 1  \\
 1 & -1& 1  \\
 1 & 1 & -1 \endbmatrix,\quad
\br=\bmatrix
3 &6 &3 \\
0 &3 &3 \\
3 &3 &0 \endbmatrix,$$
$$A'=\bmatrix
0  &6 \\
 -6 &0 \endbmatrix,\quad
B=\bmatrix
-x+z &-y &0   \\
0    &z  &x-y \endbmatrix,\quad D=\frac 16\bmatrix
0 & -x& x  \\
x & 0  &-x \\
-x& x  &0  \endbmatrix,\quad\text{and}$$
$$c_2=\frac 16\bmatrix
0            & -x^2+xz-z^2& 2x^2-xy-xz+yz \\
x^2-xz+z^2     & 0        & -x^2+xy-y^2    \\
-2x^2+xy+xz-yz &x^2-xy+y^2 & 0            \endbmatrix$$
\end{example}
 
\begin{example}
If $n=4$, then the matrices of Corollary~\ref{cor} are
$$\bp=\bmatrix
0 & 0 & 0  &0 & 0 & 0 & 5  &10& 10& 5  \\
0 & 0 & 0  &5 & 10 &10 &5  &20& 30& 20 \\
0 & 0 & 0  &10 &10 &5  &20 &30 &20& 5  \\
0 & 5 & 10 &5  &20& 30& 0 & 10 &30 &30 \\
0 & 10 &10 &20 &30& 20& 10& 30 &30& 10 \\
0 & 10 &5  &30 &20& 5 & 30& 30 &10& 0  \\
5 & 5  &20 &0  &10& 30& 0  &0  &10& 20 \\
10 &20& 30& 10 &30& 30& 0  &10 &20& 10 \\
10 &30& 20 &30 &30& 10& 10 &20 &10& 0  \\
5  &20& 5  &30 &10& 0 & 20& 10& 0 & 0  \\
\endbmatrix,$$
 $$\bp^{-1} =\frac 1{70}\bmatrix
-35& 15  &15  &-5& -10& -5 &1  &3 & 3 & 1  \\
15 & -15 &5  & 9 & 2  & -7 &-3& -3& 3  &3  \\
15 & 5   &-15& -7& 2  & 9  &3 & 3 & -3 &-3 \\
-5 & 9   &-7 & -9 &6  & 1 & 5 & -1& -3 &3  \\
-10& 2   &2  & 6 & -9 & 6 & -6& 3 & 3  &-6 \\
-5 & -7  &9  & 1 & 6   &-9& 3 & -3 &-1 &5  \\
1  & -3 & 3  & 5 & -6 & 3 & -5& 5  &-3 &1  \\
3  & -3 & 3  & -1& 3  & -3& 5  &-7& 6  &-3 \\
3  & 3  & -3 & -3& 3  & -1& -3 &6 & -7 &5  \\
1  & 3  & -3 & 3 & -6 & 5 & 1&  -3 &5  &-5 \\
\endbmatrix,$$
$$\br=\bmatrix
5 & 20 &30 &20 &5  \\
0 & 10 &30 &30 &10 \\
10 &30 &30 &10 &0  \\
0  &0  &10 &20 &10 \\
0  &10 &20 &10 &0  \\
10& 20 &10 &0  &0  \\
0 & 0  &0  &5  &5  \\
0 & 0  &5  &5  &0  \\
0 & 5 & 5  &0  &0  \\
5 & 5  &0  &0  &0  \endbmatrix,\quad A'=\frac12\bmatrix
0  & 50 & 75 & 45 \\
-50& 0  & 85 & 75 \\
-75 &-85 &0  & 50 \\
-45& -75& -50& 0  \endbmatrix,$$
$$B=\frac 12\bmatrix
-2x+2z &-x-2y &0   &0    &0     \\
0     & -x+2z& -2y &0    &0     \\
0     & 0    & 2z  &x-2y &0     \\
0      &0    & 0   &x+2z &2x-2y \endbmatrix,$$
$$D=\frac 1{70} \bmatrix
0  & -5x& 5x & -3x &x   \\
5x & 0  & -4x& 5x & -3x \\
-5x& 4x & 0  & -4x& 5x  \\
3x & -5x& 4x & 0  & -5x \\
-x & 3x & -5x& 5x & 0   \\
\endbmatrix,$$
and the matrix $c_2$ appears in Table~\ref{Table 1}.

\begin{table}
\begin{center}
{\scriptsize $$\hskip-50pt\frac 1{70}\bmatrix
0                       & -10x^2+15xz-10z^2            & 5x^2-10xy-15xz+10yz+15z^2\\   
10x^2-15xz+10z^2          & 0                          & -4x^2-5xy-10y^2-3xz-15yz-9z^2\\
-5x^2+10xy+15xz-10yz-15z^2 &4x^2+5xy+10y^2+3xz+15yz+9z^2  & 0\\                         
2x^2-15xy-16xz+15yz+17z^2  &-4x^2-4xy-15y^2-4xz-26yz-15z^2& 4x^2+3xy+9y^2+5xz+15yz+10z^2\\ 
-18x^2+17xy+17xz-17yz    & 2x^2-16xy+17y^2-15xz+15yz    & -5x^2+15xy-15y^2+10xz-10yz \endbmatrix$$} 
{\scriptsize $$\frac 1{70}\bmatrix 
      &-2x^2+15xy+16xz-15yz-17z^2   &18x^2-17xy-17xz+17yz      \\
      &4x^2+4xy+15y^2+4xz+26yz+15z^2 &-2x^2+16xy-17y^2+15xz-15yz \\
      &-4x^2-3xy-9y^2-5xz-15yz-10z^2 &5x^2-15xy+15y^2-10xz+10yz  \\
      &0                          &-10x^2+15xy-10y^2          \\
      &10x^2-15xy+10y^2             &0                        \\
\endbmatrix$$}
 \caption{The matrix $c_2$ in the resolution of $R/J$ for $J$ given (\ref{J}), when $n=4$. The top matrix is the first three columns of $c_2$. The bottom matrix is the last two columns of $c_2$.}\label{Table 1}
\end{center}
\end{table}
\end{example}
\begin{example}
If $c=6$, then the matrix $c_2$ of Corollary~\ref{cor} is  the matrix given in Table~\ref{Table 3}.
\begin{table}
\begin{center}
{\scriptsize $$\frac 1{1848}\bmatrix
0                            &  -252x^2+420xz-252z^2                 \\
252x^2-420xz+252z^2             & 0                                  \\
-126x^2+252xy+420xz-252yz-378z^2& 84x^2+126xy+252y^2+42xz+378yz+210z^2  \\
56x^2-378xy-434xz+378yz+434z^2 &  -84x^2-112xy-378y^2-56xz-644yz-350z^2 \\
-21x^2+434xy+448xz-434yz-455z^2&  54x^2+63xy+434y^2+45xz+805yz+425z^2   \\
6x^2-455xy-457xz+455yz+461z^2  &  -24x^2-24xy-455y^2-24xz-886yz-455z^2  \\
-463x^2+461xy+461xz-461yz     &  6x^2-457xy+461y^2-455xz+455yz        \\
\endbmatrix$$
$$\frac 1{1848}\bmatrix
126x^2-252xy-420xz+252yz+378z^2     & -56x^2+378xy+434xz-378yz-434z^2      \\
-84x^2-126xy-252y^2-42xz-378yz-210z^2& 84x^2+112xy+378y^2+56xz+644yz+350z^2  \\
0                                 & -60x^2-70xy-210y^2-50xz-350yz-200z^2  \\
60x^2+70xy+210y^2+50xz+350yz+200z^2  & 0                                  \\
 -75x^2-75xy-350y^2-75xz-625yz-350z^2 & 60x^2+50xy+200y^2+70xz+350yz+210z^2   \\
54x^2+45xy+425y^2+63xz+805yz+434z^2  & -84x^2-56xy-350y^2-112xz-644yz-378z^2 \\
-21x^2+448xy-455y^2+434xz-434yz     & 56x^2-434xy+434y^2-378xz+378yz       \\
\endbmatrix$$
 $$\frac 1{1848}\bmatrix
21x^2-434xy-448xz+434yz+455z^2      &-6x^2+455xy+457xz-455yz-461z^2       \\
 -54x^2-63xy-434y^2-45xz-805yz-425z^2& 24x^2+24xy+455y^2+24xz+886yz+455z^2   \\
75x^2+75xy+350y^2+75xz+625yz+350z^2  &-54x^2-45xy-425y^2-63xz-805yz-434z^2  \\
-60x^2-50xy-200y^2-70xz-350yz-210z^2 &84x^2+56xy+350y^2+112xz+644yz+378z^2  \\
0                                 &-84x^2-42xy-210y^2-126xz-378yz-252z^2 \\
84x^2+42xy+210y^2+126xz+378yz+252z^2 &0                                  \\
-126x^2+420xy-378y^2+252xz-252yz    &252x^2-420xy+252y^2                  \\
\endbmatrix$$
$$\frac 1{1848}\bmatrix
463x^2-461xy-461xz+461yz       \\
-6x^2+457xy-461y^2+455xz-455yz  \\
21x^2-448xy+455y^2-434xz+434yz  \\
-56x^2+434xy-434y^2+378xz-378yz \\
126x^2-420xy+378y^2-252xz+252yz \\
-252x^2+420xy-252y^2            \\
0                             \endbmatrix$$}
\caption{The matrix $c_2$ in the resolution of $R/J$ for $J$ given (\ref{J}), when $n=6$. The top matrix is the first two columns of $c_2$; the second matrix is columns 3 and 4 of $c_2$; the third matrix is columns 5 and 6 of $c_2$; and  the bottom matrix is the last  column of $c_2$.}\label{Table 3}
\end{center}
\end{table}
\end{example}


\begin{thebibliography}{99}

\bibitem{A57} E.~Artin, {\em Geometric algebra}, 
John Wiley \& Sons, Inc., New York, 1988.  

\bibitem{B99}M.~Boij, {\em Betti numbers of compressed level algebras}, J. Pure Appl. Algebra {\bf 134} (1999),  111--131. 

\bibitem{BE73}D.~Buchsbaum and D.~Eisenbud, {\em  What makes a complex exact{\rm?}} J. Algebra {\bf 25} (1973), 259--268. 

\bibitem{BE77}D.~Buchsbaum and D.~Eisenbud, {\em Algebra structures for finite free resolutions, and some structure theorems for ideals of codimension $3$}, Amer. J. Math. {\bf 99} (1977), 447--485.   

\bibitem{E95}D.~Eisenbud, {\em Commutative algebra. With a view toward algebraic geometry}, Graduate Texts in Mathematics, {\bf 150} Springer-Verlag, New York, 1995. 


\bibitem{EKK1}S.~El Khoury and A.~Kustin, {\em Artinian Gorenstein algebras with linear resolutions}, J. Algebra {\bf 420} (2014), 402--474. 

\bibitem{EKK2}  S.~El Khoury and A.~Kustin, {\em The explicit minimal resolution constructed from a Macaulay inverse system}, J. Algebra {\bf 440} (2015), 145--186.

\bibitem{EKK3}  S.~El Khoury and A.~Kustin, {\em  The structure of Gorenstein-linear resolutions of Artinian algebras}, J. Algebra {\bf 453} (2016), 492--560. 

\bibitem{M2}D.~Grayson and M.~Stillman,
          {\em Macaulay2, a software system for research in algebraic geometry},
          {Available at \tt{https://math.uiuc.edu/Macaulay2/}}

\bibitem{HMMNWW} T.~Harima, T.~Maeno, H.~Morita, Y.~Numata, A.~Wachi, J.~Watanabe, {\em The Lefschetz properties}, Lecture Notes in Mathematics, {\bf 2080} Springer, Heidelberg, 2013.

\bibitem{KPU17}A.~Kustin, C.~Polini, and B.~Ulrich, {\em A matrix of linear forms which is annihilated by a vector of indeterminates}, 
J. Algebra {\bf 469} (2017), 120--187.

\bibitem{KRV22} A.~Kustin, R.~R.G., and A.~Vraciu, {\em The resolution of $(x^N,y^N,z^N,w^N)$}, J. Algebra {\bf590} (2022), 338--393.

\bibitem{MMN11}
J.~Migliore, R.~Mir\'o-Roig, and U.~Nagel, {\em Monomial ideals, almost complete intersections and the weak Lefschetz property}, Trans. Amer. Math. Soc. {\bf 363} (2011),  229--257. 


\bibitem{MN} J.~Migliore and U.~Nagel, {\em Survey article: a tour of the weak and strong Lefschetz properties}, J. Commut. Algebra {\bf5} (2013),  329--358.



\bibitem{MT} R.~Mir\'o-Roig and Q.~H.~Tran, {\em The weak Lefschetz property for Artinian Gorenstein algebras of codimension three}. J. Pure Appl. Algebra {\bf 224} (2020),  106305.

\bibitem{RRR91} L.~Reid, L.~Roberts, and M.~Roitman, {\em On complete intersections and their Hilbert functions}, Canad. Math. Bull. {\bf34} (1991),  525--535.

\end{thebibliography}
\end{document}